\documentclass[10pt,a4paper]{article}

\usepackage[utf8]{inputenc}

\usepackage[backend=biber,style=alphabetic,giveninits=true,isbn=false,url=false]{biblatex}
\usepackage{csquotes}
\usepackage[T1]{fontenc}

\usepackage[scaled]{helvet} 
\usepackage{courier} 
\normalfont
\usepackage[english]{babel}
\usepackage{amsfonts}
\usepackage{amsthm}
\usepackage{amsmath}
\usepackage{amssymb}
\usepackage{a4wide}
\usepackage{mathrsfs}
\usepackage{epsfig}
\usepackage{palatino}
\usepackage{tikz,fp,ifthen}
\usepackage{float}
\usepackage{graphicx}
\usepackage[shortlabels]{enumitem}
\usepackage{babel}

\usepackage[absolute]{textpos}
\usepackage{graphicx}
\usepackage{color}
\definecolor{citegreen}{rgb}{0,0.6,0}
\definecolor{refred}{rgb}{0.8,0,0}
\usepackage[small,normal]{caption}
\usepackage{tikz}
\usepackage{authblk}
\usetikzlibrary{matrix}
\pgfdeclarelayer{background} 
\pgfsetlayers{background,main} 
\usepackage{mathtools}
\mathtoolsset{showonlyrefs}
\usepackage[colorlinks, citecolor=citegreen,linkcolor=black,hidelinks]{hyperref}
\newtheorem{theorem}{Theorem}[section]
\newtheorem{lemma}[theorem]{Lemma}
\newtheorem{proposition}[theorem]{Proposition}

\newtheorem{example}[theorem]{Example}

\theoremstyle{definition}
\newtheorem{definition}[theorem]{Definition}
\newtheorem{assumption}[theorem]{Assumption}

\theoremstyle{remark}
\newtheorem{remark}[theorem]{Remark}

\numberwithin{equation}{section}

\def\eps{\varepsilon}

\def\R{\mathbb R}

\def\R{{{\mathbb R}}}

\newcommand{\dd}{ \mathrm{d}}

\newcommand{\intbar}{\etaathop{\int\etaakebox(-13.5,0){\rule[4pt]{.7em}{0.3pt}}
\kern-6pt}\nolimits}

\newcommand{\be}{\begin{equation}}
\newcommand{\ee}{\end{equation}}
\newcommand{\bea}{\begin{equation*}}
\newcommand{\eea}{\end{equation*}}

\newcommand{\beq}{\begin{equation}}
\newcommand{\eeq}{\end{equation}}

\usetikzlibrary{backgrounds}
\usetikzlibrary{decorations.pathmorphing,backgrounds,fit,calc,through}
\usetikzlibrary{arrows}
\usetikzlibrary{shapes,decorations,shadows}
\usetikzlibrary{fadings}
\usetikzlibrary{patterns}
\usetikzlibrary{mindmap}
\usetikzlibrary{decorations.text}
\usetikzlibrary{decorations.shapes}

\usepackage{mathtools,bbm,bm}
\usepackage[normalem]{ulem}
\usepackage{cancel}

\newcommand{\cA}{\mathcal{A}}

\newcommand{\cC}{\mathcal{C}}
\newcommand{\cD}{\mathcal{D}}

\newcommand{\cG}{\mathcal{G}}
\newcommand{\cH}{\mathcal{H}}
\newcommand{\cI}{\mathcal{I}}

\newcommand{\cK}{\mathcal{K}}
\newcommand{\cL}{\mathcal{L}}

\newcommand{\cP}{\mathcal{P}}

\newcommand{\cX}{\mathcal{X}}
\newcommand{\cY}{\mathcal{Y}}

\newcommand{\bR}{\mathbb{R}}

\newcommand{\bZ}{\mathbb{Z}}

\newcommand{\bfH}{\mathbf{H}}

\newcommand{\ip}[2]{\left \langle #1,#2 \right \rangle}

\title{Well-posedness of a Hamilton-Jacobi-Bellman equation in the strong coupling regime}

\author{Serena Della Corte\thanks{Delft Institute of Applied Mathematics, Delft University of Technology, Mekelweg 4, 2628 CD Delft, The Netherlands. \emph{E-mail address}: s.dellacorte@tudelft.nl} \qquad Richard C. Kraaij\thanks{Delft Institute of Applied Mathematics, Delft University of Technology, Mekelweg 4, 2628 CD Delft, The Netherlands. \emph{E-mail address}: r.c.kraaij@tudelft.nl}}
\date{\today}

\begin{document}

\maketitle
\begin{abstract}
We prove comparison principle for viscosity solutions of a Hamilton--Jacobi--Bellman equation in a strong coupling regime considering a stationary and a time-dependent version of the equation. We consider a Hamiltonian that has a representation as the supremum of a difference of two functions: an internal Hamiltonian depending on a control variable and a function interpreted as a cost of applying the controls.
Our major innovation lies in the use of a cost function that can be discontinuous, unbounded and depending on momenta, enabling us to address previously unexplored scenarios such as cases arising from the theory of large deviations and homogenisation.
For completeness, we also state the existence of viscosity solutions and we verify the assumptions for an example arising from biochemistry.

\smallskip
\noindent\textbf{Keywords}\emph{:  Hamilton--Jacobi--Bellman equations, comparison principle, viscosity solutions, optimal control theory, Large deviations;} 

%
%
\end{abstract}



\section{Introduction}
In the present work we study well--posedness of the Hamilton--Jacobi--Bellman equation on a subset $E\subseteq\R^d$, 
\begin{align}\label{eq:Hamilton-Jacobi-Bellman}
    u(x)-\lambda \cH(x,\nabla u(x))=h(x),
\end{align}
where $\lambda$ is a positive constant and $h$ is a continuous bounded function,
and for the time--dependent version
\begin{gather}\label{eq:time-dep-HJB}
    \begin{cases}
        \partial_t u(x,t) - \cH(x,\nabla_x u(t,x)) = 0, & \text{if $t>0$,}\\
        u(0,x)=u_0(x) & \text{if $t=0$.}
    \end{cases}
\end{gather}
In the entire work we consider a Hamiltonian of the type
\begin{equation}
    \label{def:hamiltonian}
    \cH(x,p)=\sup_{\theta\in\Theta}\left[\Lambda(x,p,\theta) - \cI(x,p,\theta)\right].
\end{equation}
The main goal of this work is to prove the \textit{comparison principle} for viscosity solutions of the above equations \eqref{eq:Hamilton-Jacobi-Bellman} and \eqref{eq:time-dep-HJB}, implying also the uniqueness of solutions.
Comparison principle for viscosity solutions has been largely studied in the past years with an increasingly complex Hamiltonian. Above all, we mention \cite{BaCD97} and \cite{DLLe11} for a proof of comparison principle for equations arising from \textit{optimal control}  problems and \cite{BuDuGa18}, \cite{KuPo17} for Hamiltonians coming from the theory of large deviations for Markov processes.
In these settings, the standard assumptions used to obtain the comparison principle are usually either the \textit{modulus continuity} of $\cH$ i.e.
\begin{equation}
    \vert \cH(x,p)- \cH(y,p)\vert \leq \omega(|x-y|(1+|p|)),
\end{equation}
or uniformly coercivity of $\cH$, that is
\begin{equation}
    \sup_{x\in K} \cH(x,p) \to \infty \qquad \text{if $|p|\to \infty$.}
\end{equation}
In the case in which $\cH$ is in a variational representation as in \eqref{def:hamiltonian}, the above assumptions can be derived from conditions on $\Lambda$ and $\cI$ such that coercivity or pseudo-coercivity of $\Lambda$ and regularity and boundedness of the cost function $\cI$.
However, there is a wide class of examples violating the above assumptions. In particular, this is the case of Hamiltonians arising in the study of systems with multiple time--scales (see for example \cite{BuDuGa18}).

More recently, in \cite{KrSc21} the authors prove well--posedness for viscosity solutions of a general Hamilton--Jacobi--Bellman equation that can be applied in many of the above contexts. It is proved comparison principle under more generic and weaker assumptions then the common ones explained above. To be more precised, the authors in \cite{KrSc21} prove for the first time comparison principle for an Hamilton--Jacobi--Bellman equation with Hamiltonian of the type \eqref{def:hamiltonian} with $\Lambda$ that can be non coercive, non pseudo--coercive and non Lipschitz and $\cI$ that can be unbounded and discontinuous, but not depending on momenta $p$. Our work can then be seen as an extension of the above mentioned work as we introduce a cost function $\cI$ depending on momenta $p$. The introduction of the momenta in the function $\cI$ makes on one hand the setting even more general including examples arising from problems in homogenisation theory that could not be treated before, and on the other hand the Hamiltonian more difficult to treat as it takes into account contributions from both parts $\Lambda$ and $\cI$. For this reason, it is necessary a change of the starting assumptions based on the difference $\Lambda - \cI$ and not on the two separate functions. 

In section \ref{section:clarify-assumptions} we explain in more details how our work includes all the examples previously covered in \cite{BaCD97} and \cite{KrSc21}; in Section \ref{section:verification-for-examples-of-Hamiltonians} we give an extra example extending one of the examples presented in \cite{Po18}.
\smallskip

In the following we present a concise overview of our strategy, without delving deep into specific details.

Proving comparison principle one usually wants to bound the difference between subsolution and supersolution $\sup_E u_1- u_2$ by using a doubling variables procedure and typically ends up with an estimate of the following type 
\begin{align}
		\sup_E(u_1-u_2) \leq &\lambda \liminf_{\varepsilon\to 0}\liminf_{\alpha\to\infty} \left[\mathcal{H}\left(x_{\alpha,\eps},\dd_x \frac{\alpha}{2}d^2(x_{\alpha,\eps},y_{\alpha,\eps})\right) - \mathcal{H}\left(y_{\alpha,\eps}, - \dd_y \frac{\alpha}{2}d^2(x_{\alpha,\eps},y_{\alpha,\eps})\right)\right] \\ &+ \sup_{E}(h_1 - h_2).
  \label{eq:intro:estimate}
\end{align}
Therefore, the aim is usually to bound the difference of Hamiltonians in two sequences of points, $x_{\alpha,\eps}$ and $y_{\alpha,\eps}$, obtained as optimizers in the doubling variables procedure, and corresponding momenta $p^1_{\alpha,\eps}=\dd_x \frac{\alpha}{2}d^2(x_{\alpha,\eps},y_{\alpha,\eps})= - \dd_y \frac{\alpha}{2}d^2(x_{\alpha,\eps},y_{\alpha,\eps}) = p^2_{\alpha,\eps}$. 

Unlike the approach taken in \cite{KrSc21} and in \cite{BaCD97}, where $\Lambda$ and $\cI$ were worked on independently, in the strong coupling regime, where $\cI$ depends on $p$, we need to consider their difference as a single function. Indeed, we give assumptions on $\Lambda-\cI$. Our main assumptions are as follows:
\begin{itemize}
    \item Firstly, we rely on the \textit{continuity estimate} of $\Lambda - \cI$ that is morally the comparison principle for $\Lambda - \cI$ for fixed $\theta$. Indeed, it enables us to control the difference of $\cH$ in \eqref{eq:intro:estimate} by managing the difference of $\left(\Lambda(x_{\alpha,\eps},p^1_{\alpha,\eps},\theta_{\alpha,\eps})-\cI(x_{\alpha,\eps},p^1_{\alpha,\eps},\theta_{\alpha,\eps})\right) - \left(\Lambda(y_{\alpha,\eps},p_{\alpha,\eps}^2,\theta_{\alpha,\eps})-\cI(y_{\alpha,\eps},p^2_{\alpha,\eps},\theta_{\alpha,\eps})\right)$ for $\theta_{\alpha,\eps}$ optimizing $\cH$ and well chosen as explained in the next point.
    See Assumption \ref{assumption:regularity:Lambda-I} \ref{item:assumption:regularity:continuity_estimate} for the rigorous notions.
    \item  In order to get comparison for the Hamilton--Jacobi-Bellman equation in terms of $\cH$ by the continuity estimate for $\Lambda-\cI$, we also need to control the $\theta_{\alpha,\eps}$. For this reason, we assume the compactness of the level sets of $\cI - \Lambda$. Using this assumption we are indeed able to prove that the above sequence $\theta_{\alpha,\eps}$ is relatively compact, i.e. \ref{item:def:continuity_estimate:3}. 
    This assumption is made rigorous in Assumption \ref{assumption:regularity:Lambda-I} \ref{item:assumption:compact-sublevelsets}.
    \item We assume  $\Gamma$--convergence for $\cI-\Lambda$ to prove regularity of $\cH$. This assumption is typically true for the most treated examples, e.g. when $\Lambda$ and $\cI$ are continuous or when $\cI$ arises as a Donsker--Varadhan functional (see \cite{DoVa75a}). See Assumption \ref{assumption:regularity:Lambda-I} \ref{item:assumption:Gamma-convergence}.
\end{itemize}

To complete the well-posedness study of the Hamilton--Jacobi--Bellman equations \eqref{eq:Hamilton-Jacobi-Bellman} and \eqref{eq:time-dep-HJB}, we also state the existence of viscosity solutions for equation \eqref{eq:Hamilton-Jacobi-Bellman} in Theorem \ref{theorem:existence_of_viscosity_solution} and we mention of our ongoing work towards a new improved proof of the existence of solutions for \eqref{eq:time-dep-HJB}. To prove Theorem \ref{theorem:existence_of_viscosity_solution}, the main ingredient is to establish the existence of the differential inclusion in terms of $\partial_p \cH$. Moreover, we need to make sure that the solutions to the differential inclusions remains inside our set $E$. To this aim, we add Assumption \ref{assumption:Hamiltonian_vector_field}.

Our work is then structured as follows:
In Section \ref{section:results} we firstly give some preliminaries and an overview of the general setting and secondly we state our main results, Theorems \ref{theorem:comparison_principle_variational} and \ref{theorem:existence_of_viscosity_solution}, and the assumptions needed to prove them. In Section \ref{section:clarify-assumptions} we give a list of examples showing that our method is able to include examples treated in previous works as well as new examples that were previously beyond the scope of their applicability. Then, we prove continuity of the Hamiltonian in Section \ref{section:regularity-of-H}. In Section \ref{section:comparison_principle} we give the proof of comparison principle and we state the existence of solutions. Finally, in Section \ref{section:verification-for-examples-of-Hamiltonians} we treat an example to show that our assumptions are well-posed. 

\medskip

\textbf{Acknowledgement.}
The authors are supported by The Netherlands Organisation for Scientific Research (NWO), grant number 613.009.148 . 

\section{General setting and main results}
\label{section:results}
In this section, we firstly give some notions and definitions used throughout all the paper. Then, we proceed with the assumptions needed for the statement of our main results given in Section \ref{section:results:HJ-of-Perron-Frobenius-type}.
\subsection{Preliminaries} \label{section:preliminaries}
For a Polish space $\cX$ we denote by $C(\cX)$ and $C_b(\cX)$ the spaces of continuous and bounded continuous functions respectively. We denote $C_l(\cX)$ and $C_u(\cX)$ the spaces of lower bounded continuous and upper bounded continuous functions respectively. If $\cX \subseteq \bR^d$ then we denote by $C_c^\infty(\cX)$ the space of smooth functions that vanish outside a compact set. We denote by $C_{cc}^\infty(\cX)$ the set of smooth functions that are constant outside of a compact set in $\cX$, and by $\cP(\cX)$ the space of probability measures on $\cX$. We equip $\cP(\cX)$ with the weak topology induced by convergence of integrals against bounded continuous functions.

\smallskip

Throughout the paper, $E$ will be the set on which we base our Hamilton-Jacobi equations. We assume that $E$ is a subset of $\bR^d$ that is a Polish space which is contained in the $\bR^d$ closure of its $\bR^d$ interior. This ensures that gradients of functions are determined by their values on $E$. Note that we do not necessarily assume that $E$ is open. We assume that the space of controls $\Theta$ is Polish.

\smallskip

We next introduce \textit{viscosity solutions} for an Hamilton-Jacobi-Bellman equation $f-\lambda Af=h$ and for the time-dependent version $f_t -  A f = 0$. 
\begin{definition}[Viscosity solutions for the stationary equation] \label{definition:viscosity_solutions}
	Let $A_\dagger : \cD(A_\dagger) \subseteq C_l(E) \to C_b(E)$ be an operator with domain $\mathcal{D}(A_\dagger)$, $\lambda > 0$ and $h_\dagger \in C_b(E)$. Consider the Hamilton-Jacobi equation
	\begin{equation}
		f - \lambda A_\dagger f = h_\dagger. \label{eqn:differential_equation} 
	\end{equation}
	We say that $u$ is a \textit{(viscosity) subsolution} of equation \eqref{eqn:differential_equation} if $u$ is bounded from above, upper semi-continuous and if, for every $f \in \cD(A)$ there exists a sequence $x_n \in E$ such that
	\begin{gather*}
		\lim_{n \uparrow \infty} u(x_n) - f(x_n)  = \sup_x u(x) - f(x), \\
		\limsup_{n \uparrow \infty} u(x_n) - \lambda A_\dagger f(x_n) - h_\dagger(x_n) \leq 0.
	\end{gather*}
 Let $A_\ddagger: \cD(A_\ddagger) \subseteq C_u (E) \to C_b (E)$ be an operator with domain $\mathcal{D}(A_\ddagger)$, $\lambda > 0$ and $h_\ddagger \in C_b(E)$. Consider the Hamilton-Jacobi equation
	\begin{equation}
		f - \lambda A_\ddagger f = h_\ddagger. \label{eqn:differential_equation-ddagger} 
	\end{equation}
	We say that $v$ is a \textit{(viscosity) supersolution} of equation \eqref{eqn:differential_equation-ddagger} if $v$ is bounded from below, lower semi-continuous and if, for every $f \in \cD(A)$ there exists a sequence $x_n \in E$ such that
	\begin{gather*}
		\lim_{n \uparrow \infty} v(x_n) - f(x_n)  = \inf_x v(x) - f(x), \\
		\liminf_{n \uparrow \infty} v(x_n) - \lambda A_\ddagger f(x_n) - h_\ddagger(x_n) \geq 0.
	\end{gather*}
	We say that $u$ is a \textit{(viscosity) solution} of the set of equations \eqref{eqn:differential_equation} and \eqref{eqn:differential_equation-ddagger}, if it is both a subsolution of \eqref{eqn:differential_equation} and a supersolution of \eqref{eqn:differential_equation-ddagger}.
 \end{definition}

 \begin{definition}[Viscosity solutions for the time-dependent equation]
    Let $A_\dagger : \cD(A_\dagger) \subseteq C_l(E) \to C_b(E)$ be an operator with domain $\mathcal{D}(A_\dagger)$. Consider the Hamilton-Jacobi equation with the initial value,
    \begin{gather}
        \begin{cases}
            \partial_t u(t,x) - A_\dagger u(t,\cdot)(x)  = 0, & \text{if } t > 0, \\
            u(0,x) = u_0(x) & \text{if } t = 0,
        \end{cases} \label{eqn:HJ_def_subsolution} \\
    \end{gather}
    Let $T>0$, $f\in D(A_\dagger)$ and $g\in C^1([0,T])$ and let $F_\dagger(x,t): E\times [0,T] \to \R$ be the function 
    \begin{equation}
    F_\dagger (x,t) = 
    \begin{cases}
        \partial_t g(t) - A_\dagger f(x) & \text{if $t>0$} \\
        \left[\partial_t g(t) - A_\dagger f(x) \right] \wedge \left[u(t,x)-u_0(x) \right] & \text{if $t=0$.}
    \end{cases}
    \end{equation}
       We say that $u$ is a \textit{(viscosity) subsolution} for \eqref{eqn:HJ_def_subsolution} if for any $T > 0$ any $ f \in D(A_\dagger)$ and any $g\in C^1([0,T])$ there exists a sequence $(t_n,x_n) \in [0,T] \times E$ such that
        \begin{align*}
        & \lim_{n \uparrow \infty} u(t_n,x_n) - f(x_n) -g(t_n) = \sup_{t\in[0,T],x} u(t,x) - f(x) - g(t), \\
        & \limsup_{n \uparrow \infty} F_\dagger (x_n,t_n) \leq 0.
        \end{align*}

    Let $A_\ddagger : \cD(A_\ddagger) \subseteq C_u(E) \to C_b(E)$ be an operator with domain $\mathcal{D}(A_\ddagger)$. Consider the Hamilton-Jacobi equation with the initial value,
    \begin{gather}
        \begin{cases}
            \partial_t u(t,x) - A_\ddagger u(t,\cdot)(x)  = 0, & \text{if } t > 0, \\
            u(0,x) = u_0(x) & \text{if } t = 0,
        \end{cases} \label{eqn:HJ_def_supersolution} \\
    \end{gather}
    Let $T>0$, $f\in D(A_\ddagger)$ and $g\in C^1([0,T])$ and let $F_\ddagger(x,t): E\times [0,T] \to \R$ be the function 
    \begin{equation}
    F_\ddagger (x,t) = 
    \begin{cases}
        \partial_t g(t) - A_\dagger f(x) & \text{if $t>0$} \\
        \left[\partial_t g(t) - A_\dagger f(x) \right] \vee \left[u(t,x)-u_0(x) \right] & \text{if $t=0$.}
    \end{cases}
    \end{equation}
        
         We say that $v$ is a viscosity supersolution for \eqref{eqn:HJ_def_supersolution} if for any $T > 0$ any $f \in D(A_\ddagger)$ and $g\in C^1([0,T])$ there exists a sequence $(t_n,x_n) \in [0,T] \times E$ such that
        \begin{align*}
        & \lim_{n \uparrow \infty} u(t_n,x_n) - f(x_n) -g(t_n) = \inf_{t\in[0,T],x} u(t,x) - f(x) - g(t), \\
        & \liminf_{n\uparrow} F_\ddagger (x_n,t_n) \geq 0
        \end{align*}
        We say that $u$ is a \textit{(viscosity) solution} of the set of equations \eqref{eqn:HJ_def_subsolution} and \eqref{eqn:HJ_def_supersolution},  if it is both a subsolution of \eqref{eqn:HJ_def_subsolution} and a supersolution of \eqref{eqn:HJ_def_supersolution}.
\end{definition}
\begin{remark} \label{remark:existence of optimizers}
	Consider the definition of subsolutions for $f-\lambda A f= h$. Suppose that the testfunction $f \in \cD(A)$ has compact sublevel sets, then instead of working with a sequence $x_n$, there exists $x_0  \in E$ such that
	\begin{gather*}
		u(x_0) - f(x_0)  = \sup_x u(x) - f(x), \\
		u(x_0) - \lambda A f(x_0) - h(x_0) \leq 0.
	\end{gather*}
	A similar simplification holds in the case of supersolutions and in the case of the time-dependent equation $\partial_t f - A f = 0$.
\end{remark}
\begin{remark}
	For an explanatory text on the notion of viscosity solutions and fields of applications, we refer to~\cite{CIL92}.
\end{remark}

\begin{remark}
	At present, we refrain from working with unbounded viscosity solutions as we use the upper bound on subsolutions and the lower bound on supersolutions in the proof of Theorem \ref{theorem:comparison_principle_variational}. We, however, believe that the methods presented in this paper can be generalized if $u$ and $v$ grow slower than the containment function $\Upsilon$ that will be defined below in Definition \ref{def:containment-new}.
\end{remark}
 
We give now the definition of \textit{comparison principle} leading to a  uniqueness notion as stated in Remark \ref{remark:uniqueness}.
 \begin{definition}[Comparison Principle]
For two operators $A_\dagger, A_\ddagger \subseteq C(E) \times C(E)$, we say that the comparison principle holds if for any viscosity subsolution $u$ of $f - \lambda A_\dagger f = h_1$ (resp. $\partial_t f - A_\dagger f = 0$) and viscosity supersolution $v$ of $f - \lambda A_\ddagger f = h_2$ (resp. $\partial_t f -  A_\ddagger f = 0$), $\sup_x \left\{u(x) - v(x)\right\} \leq \sup_x \{h_1 (x) - h_2 (x)\}$ (resp. $\sup_{t\in [0,T],x} \left\{u(t,x) - v(t,x)\right\} \leq \sup_x \{ u(0,x)- v(0,x)\}$ for all $T>0$ ) holds on $E$.
\end{definition}
\begin{remark}[Uniqueness]\label{remark:uniqueness}
	If $u$ and $v$ are two viscosity solutions of~\eqref{eqn:differential_equation} or \eqref{eqn:HJ_def_subsolution}, then we have $u\leq v$ and $v\leq u$ by the comparison principle, and, hence, uniqueness of solutions.
\end{remark}

\subsection{Main results: comparison and existence}
\label{section:results:HJ-of-Perron-Frobenius-type}
In this section, we state our main results: the comparison principle, that is Theorem \ref{theorem:comparison_principle_variational}, and existence of solutions in Theorem \ref{theorem:existence_of_viscosity_solution}.
\smallskip

Consider the variational Hamiltonian $\cH : E \times \bR^d \rightarrow  \bR$ given by
\begin{equation}\label{eq:results:variational_hamiltonian}
	\mathcal{H}(x,p) = \sup_{\theta \in \Theta}\left[\Lambda(x,p,\theta) - \mathcal{I}(x,p,\theta)\right].
\end{equation}
The precise assumptions on the maps $\Lambda$ and $\mathcal{I}$ are formulated in Section~\ref{section:assumptions}. 
\begin{theorem}[Comparison principle]
	\label{theorem:comparison_principle_variational}
	Consider the map $\cH : E \times \bR^d \rightarrow \bR$ as in \eqref{eq:results:variational_hamiltonian}. Suppose that Assumption~\ref{assumption:regularity:Lambda-I} is satisfied.
	Define the operator $\bfH f(x) := \cH(x,\nabla f(x))$ with domain $\cD(\bfH) = C_{cc}^\infty(E)$. Then:
	\begin{enumerate}[(a)]
		\item \label{item:theorem-comparison} For any $h \in C_b(E)$ and $\lambda > 0$, the comparison principle holds for
		\begin{equation}\label{eq:results:HJ-eq}
			f - \lambda \, \bfH f = h.
		\end{equation}
  \item For any $f_0 \in C_b(E)$, the comparison principle holds for 
  \begin{gather}
      \begin{cases}
          \partial_t f(t,x) - \bfH f(t,\cdot)(x) = 0, & \text{if $t>0$}\\
          f(0,x)=f_0(x) &\text{if $t=0$.}
      \end{cases}
  \end{gather}
\end{enumerate}
\end{theorem}
\begin{remark}[Domain]
	The comparison principle holds with any domain that satisfies $C_{cc}^\infty(E)\subseteq \mathcal{D}(\mathbf{H})\subseteq C^1_b(E)$. We state it with $C^\infty_{cc}(E)$ to connect it with the existence result of Theorem~\ref{theorem:existence_of_viscosity_solution}, where we need to work with test functions whose gradients have compact support.
\end{remark}
Consider the Legendre dual $\cL : E \times \bR^d \rightarrow [0,\infty]$ of the Hamiltonian,
\begin{equation*}
	\cL(x,v) := \sup_{p\in\mathbb{R}^d} \left[\ip{p}{v} - \cH(x,p)\right],
\end{equation*}
and denote the collection of absolutely continuous paths in $E$ by $\cA\cC$.
\begin{theorem}[Existence of viscosity solution] \label{theorem:existence_of_viscosity_solution}
	Consider $\cH : E \times \bR^d \rightarrow \bR$ as in \eqref{eq:results:variational_hamiltonian}. Suppose that Assumption~\ref{assumption:regularity:Lambda-I} is satisfied for $\Lambda$ and~$\mathcal{I}$, and that $\mathcal{H}$ satisfies Assumption~\ref{assumption:Hamiltonian_vector_field}. For each $\lambda > 0$, let $R(\lambda)$ be the operator
	\begin{equation}\label{resolvent}
		R(\lambda) h(x) = \sup_{\substack{\gamma \in \mathcal{A}\mathcal{C}\\ \gamma(0) = x}} \int_0^\infty \lambda^{-1} e^{-\lambda^{-1}t} \left[h(\gamma(t)) - \int_0^t \mathcal{L}(\gamma(s),\dot{\gamma}(s))\right] \, \dd t.
	\end{equation}
	Then $R(\lambda)h$ is the unique viscosity solution to $f - \lambda \bfH f = h$.
\end{theorem}
\begin{remark}
	The form of the solution is typical, see for example Section III.2 in \cite{BaCD97}. It is the value function obtained by an optimization problem with exponentially discounted cost. The difficulty of the proof of Theorem~\ref{theorem:existence_of_viscosity_solution} lies in treating the irregular form of $\cH$. 
\end{remark}
\begin{remark}
    We mention that in Theorem \ref{theorem:existence_of_viscosity_solution} we only state the existence of viscosity solution for the stationary equation $f - \lambda \bfH f = h$. In our work \cite{DeCoKr23}, we will prove the existence also for the evolutionary equation.
\end{remark}
\subsection{Assumptions} \label{section:assumptions}
In this section, we formulate and comment on the assumptions imposed on the Hamiltonian defined in the previous section. 
\smallskip

We start with the \emph{continuity estimate}.
We will apply the definition below for $\cG = \Lambda-\cI$.

\begin{definition}[Continuity estimate] \label{def:results:continuity_estimate}
	Let  $\cG: E \times \mathbb{R}^d\times\Theta \rightarrow \bR$, $(x,p,\theta)\mapsto \cG(x,p,\theta)$ be a function. Suppose that for each $\varepsilon > 0$, 
	there is a sequence of positive real numbers $\alpha \rightarrow \infty$.
	
	Suppose that for each $\varepsilon$ and $\alpha$ we have variables $(x_{\alpha,\eps},y_{\alpha,\eps})$ in $E^2$ and variables $\theta_{\alpha,\eps}$ in $\Theta$. We say that this collection is \textit{fundamental} for $\cG$ with if:
	\begin{enumerate}[label = (C\arabic*)]
		\item \label{item:def:continuity_estimate:1} For each $\varepsilon$, there are compact sets $K_\varepsilon \subseteq E$ and $\widehat{K}_\varepsilon\subseteq\Theta$ such that for all $\alpha$ we have $x_{\alpha,\eps},y_{\alpha,\eps} \in K_\varepsilon$ and $\theta_{\alpha,\eps}\in\widehat{K}_\varepsilon$.
		\item \label{item:def:continuity_estimate:2} 
		For each $\varepsilon > 0$, we have $\lim_{\alpha \rightarrow \infty} \frac{\alpha}{2} d^2(x_{\alpha,\eps},y_{\alpha,\eps}) = 0$. For any limit point $(x_\varepsilon,y_\varepsilon)$ of $(x_{\alpha,\eps},y_{\alpha,\eps})$, we have $\dd(x_{\varepsilon},y_{\varepsilon}) = 0$.
		\item \label{item:def:continuity_estimate:3} We have 
		for all $\varepsilon > 0$
		\begin{align} 
			& \sup_{\alpha} \cG\left(y_{\alpha,\eps}, - \frac{\alpha}{2}\dd_y d^2 (x_{\alpha,\eps},\cdot)(y_{\alpha,\eps}),\theta_{\alpha,\eps}\right) < \infty, \label{eqn:control_on_Gbasic_sup} \\
			& \inf_\alpha \cG\left(x_{\alpha,\eps}, \frac{\alpha}{2} \dd_x d^2(\cdot,y_{\alpha,\eps})(x_{\alpha,\eps}),\theta_{\alpha,\eps}\right) > - \infty. \label{eqn:control_on_Gbasic_inf} 	
		\end{align} \label{itemize:funamental_inequality_control_upper_bound}
	\end{enumerate}
	We say that $\cG$ satisfies the \textit{continuity estimate} if for every fundamental collection of variables we have for each $\varepsilon > 0$ that
	\begin{equation}\label{equation:Xi_negative_liminf}
		\liminf_{\alpha \rightarrow \infty} \cG\left(x_{\alpha,\eps}, \alpha  \dd_x\frac{1}{2}d^2(\cdot,y_{\alpha,\eps})(x_{\alpha,\eps}),\theta_{\alpha,\eps}\right)
		- \cG\left(y_{\alpha,\eps}, - \alpha \dd_y\frac{1}{2}d^2(x_{\alpha,\eps},\cdot)(y_{\alpha,\eps}),\theta_{\alpha,\eps}\right) \leq 0.
	\end{equation}
\end{definition}
The continuity estimate is indeed exactly the estimate that one would perform when proving the comparison principle for the Hamilton-Jacobi equation in terms of the Hamiltonian \eqref{eq:results:variational_hamiltonian} 
 (disregarding the supremum over $\theta$). Indeed, in standard proofs of comparison principle one usually wants to control the difference of Hamiltonians calculated in particular collections of points.
Typically, the control on $(x_{\alpha,\eps},y_{\alpha,\eps})$ that is assumed in \ref{item:def:continuity_estimate:1} and \ref{item:def:continuity_estimate:2} is obtained from choosing $(x_{\alpha,\eps},y_{\alpha,\eps})$ as optimizers in the doubling of variables procedure (see Lemma~\ref{lemma:doubling_lemma}), and the control that is assumed in~\ref{item:def:continuity_estimate:3} is obtained by using the viscosity sub- and supersolution properties in the proof of the comparison principle. The required restriction to compact sets in Lemma~\ref{lemma:doubling_lemma} is obtained by including in the test functions a \emph{containment function}.
\begin{definition}[Containment function]\label{def:containment-new}
    We say that a function $\Upsilon: E\to [0,\infty]$ is a \textit{containment function} for $\cH$ if $\Upsilon\in C^1(E)$ and there exists a constant $c_\Upsilon$ such that
    \begin{itemize}
        \item For every $c\geq 0$, the set $\{x\, | \, \Upsilon(x)\leq c\}$ is compact;
        \item $\sup_\theta \sup_x \left(\Lambda(x,\nabla\Upsilon(x), \theta)- \cI (x, \nabla \Upsilon(x), \theta)\right)\leq c_\Upsilon. $
    \end{itemize}
\end{definition}
For two constants $M_1, M_2$ and a compact set $K\subseteq E$ we write
\begin{align}\label{eq:def_levelsets}
    \Theta_{M_1,M_2,K} := \bigcup_{x, y\in K} \bigcup_{\alpha > 1} \biggl\{\theta\in \Theta \, |  \, \cI\left(x,\partial_x \frac{\alpha}{2} d^2 (x,y),\theta\right) &-\Lambda\left(x,\partial_x \frac{\alpha}{2} d^2 (x,y),\theta\right)\leq M_1 , \\  
    \Lambda\left(y, \partial_y \frac{- \alpha}{2} d^2 (x,y), \theta\right) &- \cI\left(y,\partial_y \frac{-\alpha}{2} d^2 (x,y),\theta\right)\leq M_2\biggr\}.
\end{align}
To prove the main results we will also make use of the continuity of $\cH$. The continuity of $\cH$ is proved in Proposition \ref{prop:reg-of-H-and-L:continuity} by making use of the following notion of convergence for the function $\cI-\Lambda$.
\begin{definition}[$\Gamma$--convergence]\label{def:Gamma-convergence}
Let $J:E\times\R^d\times \Theta\to \R\cup\{\infty\}$. We say that $J$ is $\Gamma$--convergent in terms of $(x,p)$, if 
\begin{enumerate}
    \item If $x_n\to x$ in $E$, $p_n\to p $ in $\R^d$ and $\theta_n\to \theta$ then $\liminf_n J(x_n,p_n,\theta_n)\geq J(x,p,\theta)$,
    \item For $x_n\to x$ and $p_n\to p$ and for all $\theta\in\Theta$ there are $\theta_n$ such that $\theta_n\to \theta$ and $\limsup_n J(x_n,p_n,\theta_n)\leq J(x,p,\theta)$.
\end{enumerate}
    
\end{definition}
We will consider the following assumption.
\begin{assumption}\label{assumption:regularity:Lambda-I}
The functions $\cH$ and $\Lambda-\cI$ verify the following properties.
    \begin{enumerate}[(I)]
        \item \label{item:assumption:convexity}The map $p\mapsto\cH(x,p)$ is convex and $\cH(x,0)=0$ for every $x\in E$.
        \item \label{item:assumption:regularity:boundness}The function $\theta\mapsto\Lambda(x,p,\theta)-\cI(x,p,\theta)$ is bounded from above for every $x,p$.
        \item \label{item:assumption:compact_containment}There exists a containment function $\Upsilon:E\to[0,\infty]$ in the sense of Definition \ref{def:containment-new}.
        \item \label{item:assumption:Gamma-convergence} The function $\cI-\Lambda$ is $\Gamma$--convergent in terms of $(x,p)$.
        \item \label{item:assumption:compact-sublevelsets} $\forall M_1, M_2 \in \R$ and $ \forall K$ compact, the set $\Theta_{M_1,M_2,K}$ is relatively compact.
        \item \label{item:assumption:statinary_meas}
        $\forall x \in E$ , $\forall p \in \R^d$ and for all small neighborhood of $x$ and $p$, $U_x$ $V_p$, there exists a continuous function $g:U_x\times V_p \to \R$ such that the set $\phi_g(y,q)=\{\theta \in \Theta \, | \, \cI(y,q,\theta) - \Lambda (y,q,\theta) \leq g(y,q)\}$ is non-empty and compact $\forall \, y\in U_x, q \in V_p$.
        \item \label{item:assumption:regularity:continuity_estimate}The function $\Lambda - \cI$ verifies the continuity estimate in the sense of Definition \ref{def:results:continuity_estimate}.
     \end{enumerate}
\end{assumption}
To establish the existence of viscosity solutions, we will impose one additional assumption. For a general convex functional $p \mapsto \Phi(p)$ we denote
\begin{equation} \label{eqn:subdifferential}
	\partial_p \Phi(p_0)
	:= \left\{
	\xi \in \mathbb{R}^d \,:\, \Phi(p) \geq \Phi(p_0) + \xi \cdot (p-p_0) \quad (\forall p \in \mathbb{R}^d)
	\right\}.
\end{equation}

\begin{definition} \label{definition:tangent_cone}
	The tangent cone (sometimes also called \textit{Bouligand cotangent cone}) to $E$ in $\bR^d$ at $x$ is
	\begin{equation*}
		T_E(x) := \left\{z \in \bR^d \, \middle| \, \liminf_{\lambda \downarrow 0} \frac{d(x + \lambda z, E)}{\lambda} = 0\right\}.
	\end{equation*}
\end{definition}

\begin{assumption} \label{assumption:Hamiltonian_vector_field} 
One of the following conditions hold:
 \begin{enumerate}[(a)]
     \item the set $E$ is open;
     \item the set $E$ is closed and convex and the map $p\mapsto \cH(x,p)$ is such that $\partial_p \cH(x,p) \cap T_E(x) \neq \emptyset$ for all $x \in E$, $p \in \bR^d$.
 \end{enumerate}
 
\end{assumption}
The above assumption implies that the solutions of the differential inclusion in terms of $\partial_p \cH(x,p)$ remain inside $E$.

\section{Basic examples to clarify the assumptions and comparison with previous works}\label{section:clarify-assumptions}
Although in Section \ref{section:verification-for-examples-of-Hamiltonians} we present a fully new example beyond present results, in this section, we give four examples. The first two are those treated in the previous works \cite{BaCD97}  \cite{KrSc21}, showing that our results includes these examples. To emphasize the fact that we can address more examples then the previous works, we illustrate a third example that does not fall into the list of examples considered in \cite{KrSc21} but that can be studied with our results. The final example clarifies  Assumption \ref{assumption:regularity:Lambda-I} \ref{item:assumption:compact-sublevelsets}.
\subsection{The classical control Hamiltonian}
The following first example is a classical example that was firstly given in \cite{BaCD97}. 
\begin{example}
Consider the Hamiltonian 
\begin{equation}
    \cH(x,p) = \sup_{a\in A} \left\{ - f(x,a) \cdot p - l(x,a)\right\},
\end{equation}
with $A$ a compact set, $f$ a Lipschitz continuous function and $l$ a non-negative continuous function such that 
\begin{equation}
    \sup_{a\in A}\vert l(x,a) - l(y,a)\vert \leq w ( \vert x-y \vert),
\end{equation}
with $w$ a modulus of continuity. Then, $\cH$ verifies Assumption \ref{assumption:regularity:Lambda-I}.
\end{example}
\begin{proof}[Proof sketch.]
We show in the following that the example verifies our assumptions.
\begin{enumerate}[(I)]
    \item $\cH(x,p)$ is convex in $p$ and $\cH(x,0)=\sup_{a\in A} \left\{ - l(x,a)\right\}= 0$.
    \item $a\mapsto \Lambda(x,p,a) - \cI(x,a)$ is bounded since $A$ is compact and $\Lambda - \cI$ is continuous.
    \item $\Upsilon(x) = \frac{1}{2} \log (1 + \vert x \vert ^2)$ is a containment function, in the sense of Definition \ref{def:containment-new}, since $f$ is Lipschitz.
    \item $\cI - \Lambda$ is continuous and hence $\Gamma$- convergent in $(x,p)$.
    \item For every $M_1,M_2 \in \R$ and $K$ compact, the closure of $\Theta_{M_1,M_2,K}$ is a closed subset of $A$ and hence it is compact. 
    \item Since $\cI(x,p,\cdot) - \Lambda(x,p,\cdot)$ is a continuous function on $A$ compact, there exists $M\geq 0$ such that $\cI(x,p,a) - \Lambda(x,p,a)\leq M$, for every $a\in A$.
    \item The function $\Lambda - \cI$ verifies trivially the continuity estimate due to the Lipschitz property of $f$ and the modulus continuity of $l$. 
\end{enumerate}
\end{proof}

We can then conclude that our result extends \cite[Theorem 3.1]{BaCD97}. In the same way it is possible to show that our result extends also the time--dependent case \cite[Theorem 3.7]{BaCD97}.

\smallskip
\subsection{Hamiltonians with discontinuous and x dependent cost function}
As explained in the introduction, cases in which $\cI$ is not bounded or not continuous and $\Lambda$ is not Lipschitz or not coercive are not covered by \cite{BaCD97}. In \cite{KrSc21}, the authors instead treat these cases, but keeping the cost function $\cI$ independent of $p$. 
We would like to emphasize that while many of our assumptions are implied by the assumptions in \cite{KrSc21}, it is not straightforward and clear whether our fifth assumption can be derived from the assumptions in \cite{KrSc21} (in particular by their assumption ($\cI 3$)). The inclusion of momentum in the cost function adds an element of complexity, making it challenging to generalize the previous assumption that does not account for momentum. Nevertheless, our approach includes all the examples addressed in \cite{KrSc21} and more. 

In particular, we show in the following two examples of Hamiltonians with a discontinuous and $x$ dependent cost function $\cI$. The first one was treated in Proposition 5.9 in \cite{KrSc21} and in Example \ref{example:x_dep_cost-function_workingforkraaijsch} we show that it also matches our assumptions. The second example is a case of Hamiltonian that can not be addressed by \cite{KrSc21}. But in Example \ref{example:x_dep_cost-function_workingforus} we show that it does fit our assumptions and hence that we can cover more cases then previous works.

\begin{example}\label{example:x_dep_cost-function_workingforkraaijsch}
We consider $E=\R^d$, $F= \{1, \dots, J\}$ and $\Theta=\mathcal{P}(F)$ and
\begin{align}
    &\Lambda(x,p,\theta) = \sum_{i\in F} \left[\langle a(x,i) p , p \rangle + \langle b(x,i), p \rangle \right]\theta_i \\ 
    &\cI(x,\theta) = \sup_{w\in \R^J} \sum_{ij} r(i,j,x) \theta_i \left [ 1 - e^{w_j - w_i}\right],
\end{align}
where $a: E \times F \to \R^{d\times d}$, $b: E\times F \to \R^d $, $ r: F^2\times E \to [0,\infty)$ and $\theta_i = \theta(\{i \})$. Then, Assumption \ref{assumption:regularity:Lambda-I} holds. 
\end{example}
\begin{proof}[Proof Sketch.] 
\begin{enumerate}[(I)]
    \item Convexity of $\cH$ in $p$ and the fact that $\cH(x,0)=0$ is shown in \cite{KrSc21} Proposition 5.11. 
    \item $\theta \mapsto \Lambda(x,p,\theta) - \cI(x,p,\theta)$ is bounded since $\Lambda$ is bounded and continuous and $\cI$ has compact sublevel sets in $\Theta$ as shown in Proposition 5.9 of \cite{KrSc21}.
    \item The existence of a containment function $\Upsilon$ is shown in Proposition 5.11 of \cite{KrSc21}. Indeed, in \cite{KrSc21} the authors show that there exists a function $\Upsilon$ that has compact sublevel sets and such that $ \Lambda(x,\nabla \Upsilon(x),\theta)$ is bounded from above. By observing that $\Lambda(x,\nabla \Upsilon(x), \theta) - \cI(x,\theta)$ is bounded above by $\Lambda(x,\nabla \Upsilon(x),\theta)$, it becomes evident that $\Upsilon$ is also a containment function in the sense of Definition \ref{def:containment-new}.
    \item $\cI - \Lambda$ is $\Gamma$-convergent as $\cI$ is $\Gamma$-convergent (see Proposition 5.9 in \cite{KrSc21}) and $\Lambda$ is continuous.
    \item For every $M_1,M_2\in \R$ and every $K$ compact, the closure of $\Theta_{M_1,M_2,K}$ is a closed subset of $\Theta$ and, hence, compact. 
    \item In Proposition 5.9 of \cite{KrSc21} it is shown that there exists a $\theta^0(x,p)$ such that $\cI(x,p,\theta^0(x,p))=0$. Hence, taking $g(x,p)=-\Lambda(x,p,\theta^0(x,p))$, it follows that the set $\phi_{g}(x,p)$ is not empty.
    \item $\Lambda - \cI$ verifies the continuity estimate since $\Lambda$ verifies the continuity estimate as shown in Proposition 5.11 in \cite{KrSc21} and $\cI$ is equicontinuous by Proposition 5.9 in \cite{KrSc21}.
    \end{enumerate}
\end{proof}
   The following example is a case of Hamiltonian with an unbounded internal Hamiltonian $\Lambda$ in terms of $\theta$. This leads to a situation where Assumption $(\Lambda 4)$ in \cite{KrSc21} fails. In the following we prove the validity of our Assumption \ref{assumption:regularity:Lambda-I}. 

The key distinction in our approach, enabling the success of this example while it remained unaddressed in \cite{KrSc21}, lies in our treatment of $\Lambda$ and $\cI$ in an integrated whole $\Lambda - \cI$ rather than separately and with two different types of assumptions. In this case,  while $\Lambda$ is not bounded in terms of $\theta$, the composite $\Lambda - \cI$ is bounded. 
Consequently, in contrast to Assumption $(\Lambda 4)$ in \cite{KrSc21}, which solely considers the boundedness of $\Lambda$ and subsequently fails in this scenario, our Assumption \ref{assumption:regularity:Lambda-I}  \ref{item:assumption:regularity:boundness} instead requires the boundedness of $\Lambda - \cI$ and consequently holds.
\begin{example}\label{example:x_dep_cost-function_workingforus}
        Consider $E=\R^d$ and $\Theta = \mathcal{P}(\R)$ and 
        \begin{align}
            &\Lambda(x,p,\theta) = \int_\R - x^3 p + \frac{1}{2}(1+|z|) p^2 \, \theta(dz) \\
            &\cI(x,\theta) = - \inf_{\phi \in C^2(\R)} \int_\R \frac{L_x \phi(z)}{\phi(z)} \, \theta(dz),
        \end{align}
        where
        \begin{equation}
            L_x \phi(z) := - (z - x) \phi ' (z) + \frac{1}{2}\Delta \phi(z) .
        \end{equation}
Then, Assumption \ref{assumption:regularity:Lambda-I} holds.
\end{example}
\begin{proof}[Proof sketch.]

In the following we prove Assumptions \ref{item:assumption:compact_containment} and \ref{item:assumption:compact-sublevelsets}. The proof of the other assumptions are or trivial or similar to the proof for Example \ref{example:x_dep_cost-function_workingforkraaijsch}.

We firstly prove that 
\begin{equation}\label{eq:costfunction_bound}
    \cI(x,\theta) \geq -\frac{1}{2} (1+ x^2) + \frac{1}{2}\int_\R (x-z)^2 \theta(dz),
\end{equation}
then, we will proceed with the proof of Assumptions \ref{item:assumption:compact_containment} and \ref{item:assumption:compact-sublevelsets}.

\textit{Proof of \eqref{eq:costfunction_bound}.}
Let $\tilde{\phi}(z) = e^{1+\frac{1}{2} z^2}$ and $\xi_n(r) = \frac{n r}{n+r}$ and let $\phi_n = \xi_n \circ \tilde{\phi}$ and call $J(x,\theta)$ the function at the right-hand side of \eqref{eq:costfunction_bound}. Note that 
\begin{equation}
    J(x,\theta) = - \int_\R \frac{L_x \tilde{\phi}}{\tilde{\phi}}.
\end{equation}
Moreover, 
\begin{equation}
    \inf_g \int_\R \frac{L_x g}{g} \leq \liminf_n \int_\R \frac{L_x \phi_n}{\phi_n}.
\end{equation}
Consequently, we only need to prove that 
\begin{equation}
    \liminf_n \int_\R \frac{L_x \phi_n}{\phi_n} \leq \int_\R \frac{L_x \tilde{\phi}}{\tilde{\phi}}.
\end{equation}
By definition of $L_x$, we have
\begin{align}\label{eq:generator_testfun} 
L_x \phi_n(z) = & - (z-x) \left(\frac{n^2}{(n+e^{1+1/2 z^2})^2} \right) \tilde{\phi}'(z) + \frac{1}{2} \left(\frac{n^2}{(n+e^{1+1/2 z^2})^2} \right) \tilde{\phi}''(z) \\
    & + \frac{1}{2}\left[\frac{-2n^2}{(n+ e^{1+1/2 z^2})^3} z^2 e^{2(1+1/2 z^2)} \right].
\end{align}
Dividing \eqref{eq:generator_testfun} by $\phi_n$, 
\begin{equation}
    \frac{L_x\phi_n}{\phi_n} = \left(\frac{n^2}{(n+e^{1+1/2 z^2})^2} \right) \frac{L_x \tilde{\phi}}{\phi_n} - \left(\frac{n^2}{(n+e^{1+1/2 z^2})^3} \right) z^2 e^{2(1+1/2 z^2)}.
\end{equation}
Moreover, 
\begin{align}
    \liminf_n \int_\R \frac{L_x \phi_n}{\phi_n}
    &= \liminf_n \int_\R \underbrace{\left(\frac{n^2}{(n+e^{1+1/2 z^2})^2} \right) \frac{L_x \tilde{\phi}}{\phi_n}}_{\downarrow \frac{L_x \tilde{\phi}}{\tilde{\phi}}}
    + \liminf_n \int_\R \underbrace{-\left(\frac{n^2}{(n+e^{1+1/2 z^2})^3} \right) z^2 e^{2(1+1/2 z^2)}}_{\leq 0} \\
    &\leq \int_\R \frac{L_x \tilde{\phi}}{\tilde{\phi}},
\end{align}
where in the last inequality we used that $\xi_n (r)$ converges to $r$ for $n \to \infty$. We can conclude that
\begin{align}
    \cI(x,\theta) \geq \int_\R -(z-x) z + \frac{1}{2} (1+z^2) = \int_\R - \frac{1}{2}(1+x^2) + \frac{1}{2}(x - z)^2 \, \theta(dz),
\end{align}
establishing \eqref{eq:costfunction_bound}.

\textit{Proof of Assumption \ref{item:assumption:compact-sublevelsets}}.
We want to prove that $\Theta_{M_1,M_2,K}$ defined in \eqref{eq:def_levelsets} is relatively compact for all $M_1,M_2$ constants and $K$ compact. 
To this aim, note that, by definition of $\Theta_{M_1,M_2,K}$ ( \eqref{eq:def_levelsets} at page \pageref{eq:def_levelsets}) and by \eqref{eq:costfunction_bound}, we have
\begin{align}\label{eq:subseteq_levelset}
    \Theta_{M_1,M_2,K} &\subseteq \bigcup_{x,y\in K} \bigcup_{\alpha>1} \left\{ \theta \, | \, \cI(x,\theta) - \Lambda(x,\partial_x\frac{\alpha}{2} d^2(x,y),\theta) < M_1\right\} \\ &\subseteq \bigcup_{x,y\in K} \bigcup_{\alpha>1} \left\{ \theta \, | \, J(x,\theta) - \Lambda(x,\partial_x\frac{\alpha}{2} d^2(x,y),\theta) < M_1\right\} := \Theta^*_{M_1,K}.
\end{align} 

Moreover, since $J(x,\theta)$ has compact sublevel sets in $\theta$ for $x$ in a compact $K$, and considering the quadratic nature of $J$  together with the linearity of $\Lambda$ in $z$, it follows that the sublevel sets of $J - \Lambda$ are also compact in $\theta$.
Consequently, this implies that the set $\Theta^*_{M_1,K}$ is compact. Then, by \eqref{eq:subseteq_levelset}, we can conclude that $\Theta_{M_1,M_2,K}$ is relatively compact.

\textit{Proof of Assumption \ref{item:assumption:compact_containment}.} We prove that $\Upsilon(x) = \frac{1}{2} x^2$ is a containment function.
To do so, note that by \eqref{eq:costfunction_bound} we have
\begin{align}
    \Lambda(x,\nabla \Upsilon(x),\theta) - \cI(x,\theta) &\leq - x^4 + \frac{1}{2}\int_\R (1+|z|) x^2 \, \theta(dz) + \frac{1}{2} (1+ x^2) - \frac{1}{2}\int_\R (x-z)^2 \, \theta(dz) \\
    &= x^2 - x^4 + \frac{1}{2} + \frac{1}{2} \int_\R |z| x^2 -\frac{1}{2} (x-z)^2 \theta(dz)\\
    &= x^2 - x^4 + \frac{1}{2} +  \int_\R \frac{1}{2}|z| x^2 -\frac{1}{2} x^2 - \frac{1}{2} z^2 +xz \, \theta(dz)\\
    &= \frac{1}{2}x^2 - x^4 + \frac{1}{2} +  \int_\R \frac{1}{2}|z| x^2  - \frac{1}{2} z^2 +\left(x\sqrt{2}\right)\left(\frac{z}{\sqrt{2}}\right) \, \theta(dz) \\
    & \leq \frac{1}{2} x^2 - x^4 + \frac{1}{2} + \int_\R \frac{1}{4} z^2 + \frac{1}{4} x^4 - \frac{1}{2} z^2 + x^2 + \frac{1}{4}z^4 \, \theta(dz)\\
    &= - \frac{3}{4} x^4 + \frac{3}{2}x^2 + \frac{1}{2},
\end{align}
where, in the last inequality we used the rule $ab \leq \frac{1}{2}a^2 +\frac{1}{2}b^2$ for $\frac{1}{2}|z|x^2$ and $xz$. Note that the last line is bounded from above for all $x$ and $\theta$. This concludes the proof.
\end{proof}
\subsection{Hamiltonian with x and p dependent cost function}
    Although \cite{KrSc21} effectively covers a greater number of cases than \cite{BaCD97}, as explained in the previous subsection, it is unable to address scenarios where the cost function depends on momenta $p$. For instance, in \cite{DecoKr} and \cite{Po18}, it is proved comparison principle with two different "ad hoc" proofs involving coercivity or Lipschitz estimates and optimization problems, for an Hamilton--Jacobi--Bellman equation with an Hamiltonian of the type as in \eqref{def:hamiltonian}. 
These examples does not fall into the cases of \cite{KrSc21} due to the presence of $p$ in $\cI$. Indeed, assumptions as $(\cI 5) $ and $(\cI 4)$ in \cite{KrSc21} are not satisfied and quite challenging to modify to incorporate the momenta $p$. Our assumptions are instead satisfied.

  We also want to mention that in \cite{FK06} it is possible to find some examples of Hamiltonians with cost function depending on $p$. We confidently assert that our results can cover all these examples.
  
We can conclude, then, that our work can be used for a large class of examples including examples that fall into previously treated theories, as the two examples above, as well as those that have been addressed by means of "ad hoc" proofs, as the examples in \cite{DecoKr} and \cite{Po18}, and cases left unexplored as the one that we study in Section \ref{section:verification-for-examples-of-Hamiltonians}.

\smallskip
\subsection{Examples to clarify Assumption \ref{item:assumption:compact-sublevelsets}}
In this subsection, we want to clarify our Assumption \ref{item:assumption:compact-sublevelsets} using two straightforward examples. 

In the introduction we explained that the first three assumptions are standard in the proof of comparison principle and that the $\Gamma$ - convergence and assumption \ref{item:assumption:statinary_meas} is needed to prove continuity of the Hamiltonian. We show now that Assumption \ref{item:assumption:compact-sublevelsets} is a key assumption. Indeed, the continuity estimate of $\Lambda - \cI $ is morally equivalent to the comparison principle for $\Lambda - \cI$ for well chosen $\theta$. Thus, one is enable to control the difference of Hamiltonians, and hence to prove comparison principle for $\cH$ by controlling the difference of $\Lambda- \cI$, if it is also possible to control the variable $\theta$, i.e., if assumption \ref{assumption:regularity:Lambda-I} \ref{item:assumption:compact-sublevelsets} holds. In the first example we will show how assumption 
\ref{assumption:regularity:Lambda-I} \ref{item:assumption:compact-sublevelsets} may not hold and subsequently lead to the failure of the comparison principle. In the second example, we show how to get Assumption \ref{assumption:regularity:Lambda-I} \ref{item:assumption:compact-sublevelsets} and, hence, comparison.
\begin{example}\label{example:assumptionV_notworking}
Consider $E= \R^2$, $\Theta = \R$ and for $x=(x_1,x_2)\in E$ and $p=(p_1,p_2) \in \R^2$ consider
\begin{equation}
    \Lambda(x,p,\theta)= \frac{1}{2} a(x) p_1 ^2 + b(x,\theta) p_1,
\end{equation}
and 
\begin{equation}
    \cI(p,\theta) = |p_2 - \theta|.
\end{equation}
Then, Assumption \ref{item:assumption:compact-sublevelsets} fails.
\end{example}
\begin{proof}[Proof sketch.]
In a standard comparison principle proof, one usually finds two sequences of points $x_{\alpha,\eps}$ and $y_{\alpha,\eps}$, optimizers in a doubling variable procedure, and the corresponding momenta $p_{\alpha,\eps}$. Assuming that, for some $\theta_{\alpha,\eps}$ and all fixed small $\eps >0$, 
\begin{equation}
   \liminf_{\alpha\to\infty}( \Lambda - \cI )( x_{\alpha,\eps}, p_{\alpha,\eps}, \theta_{\alpha,\eps}) - ( \Lambda - \cI )( y_{\alpha,\eps}, p_{\alpha,\eps}, \theta_{\alpha,\eps}) \leq 0,
\end{equation}
i.e., that morally comparison holds for $\Lambda - \cI$ along $\theta_{\alpha,\eps}$, one wants then use this bound to prove comparison for $\cH$. However, for $p_1\neq p_2$, Assumption \ref{assumption:regularity:Lambda-I} \ref{item:assumption:compact-sublevelsets} fails since it is not possible to control the $\theta_{\alpha,\eps}$. 
Indeed, using the subsolution and supersolution properties, one is able to show that 
\begin{align}\label{example:subsolution}
   \sup_\alpha (\cI - \Lambda)(x_{\alpha,\eps}, p_{\alpha,\eps}, \theta_{\alpha,\eps}) < \infty,
\end{align}
\begin{equation} \label{example:supersolution}
    \sup_\alpha (\Lambda - \cI) (y_{\alpha,\eps}, p_{\alpha,\eps}, \theta_{\alpha,\eps}) < \infty.
\end{equation}

However, with the above bounds, in this case, we can not conclude any control on the the sequence $\theta_{\alpha,\eps}$. Indeed, if for example $p_{2,\eps,\alpha}$ blows up, $\theta_{\alpha,\eps}$ consequentially does the same in order to have \eqref{example:subsolution}. Using this observation, we can conclude that comparison principle in this case fails.
\end{proof}
In the following example we show how to make Assumption \ref{item:assumption:compact-sublevelsets} work.
\begin{example}\label{example:assumptionV_working}
 We consider $E=\R$, $\Theta = \R$ and for $x\in E$ and $p\in \R$ consider
 \begin{equation}
     \Lambda(x,p,\theta)=\frac{1}{2} a(x) p^2 + b(x,\theta)p,
 \end{equation}
 and
 \begin{equation}
     \cI(p,\theta) = |p-\theta|.
 \end{equation}
 Assumption \ref{item:assumption:compact-sublevelsets} holds.
 \end{example}
 \begin{proof}[Proof sketch.]
The Hamiltonian now is coercive as $\Lambda$ controls the order of $p$. This means that 
$$\lim_{|p| \to \infty} \Lambda(x,p,\theta) - \cI (x,p,\theta) = \infty. $$
In this case, then, by coercivity and \eqref{example:supersolution}, $p_{\alpha,\eps}$ can not blow up. Using this fact and the bound \eqref{example:subsolution}, we can conclude that $\theta_{\alpha,\eps}$ can not go to infinity as well.
 In this case assumption \ref{assumption:regularity:Lambda-I} \ref{item:assumption:compact-sublevelsets} holds and comparison as well.  
\end{proof}

\begin{remark}
    Note that the Hamiltonian in Examples \ref{example:assumptionV_notworking} and \ref{example:assumptionV_working} is not convex in $p$. For the sake of clarity, we used this example as it is easy to see the role of Assumption \ref{item:assumption:compact-sublevelsets}. Replacing the above $\cI$ with the Donsker-Varadhan rate function for the Ornstein--Uhlenbeck diffusion process centered in $p_2$, one does get a convex Hamiltonian for which the above argument applies.
\end{remark}

\section{Regularity of the Hamiltonian}\label{section:regularity-of-H}
We start showing that the Hamiltonian is continuous. This is the content of the following proposition.
\begin{proposition}[Continuity of the Hamiltonian]\label{prop:reg-of-H-and-L:continuity}
    Let $\mathcal{H} : E \times \mathbb{R}^d\to \mathbb{R}$ be the Hamiltonian defined in~\eqref{eq:results:variational_hamiltonian}, and suppose that Assumption~\ref{assumption:regularity:Lambda-I} is satisfied. Then the map $(x,p) \mapsto \cH(x,p)$ is continuous and the Lagrangian $(x,v) \mapsto \cL(x,v) := \sup_{p} \ip{p}{v} - \mathcal{H}(x,p)$ is lower semi-continuous.
\end{proposition}
We will use the following technical result to establish upper semi-continuity of $\cH$.
	\begin{lemma}[Lemma 17.30 in \cite{AlBo06}] \label{lemma:upper_semi_continuity_abstract}
		Let $\cX$ and $\cY$ be two Polish spaces. Let $\phi : \cX \rightarrow \cK(\cY)$, where $\cK(\cY)$ is the space of non-empty compact subsets of $\cY$. Suppose that $\phi$ is upper hemi-continuous, that is if $x_n \rightarrow x$ and $y_n \rightarrow y$ and $y_n \in \phi(x_n)$, then $y \in \phi(x)$. 
		
		Let $f : \text{Graph} (\phi) \rightarrow \bR$ be upper semi-continuous. Then the map $m(x) = \sup_{y \in \phi(x)} f(x,y)$ is upper semi-continuous.
	\end{lemma} 
	\begin{proof}[Proof of Proposition~\ref{prop:reg-of-H-and-L:continuity}]
We start by establishing the upper semi-continuity arguing on the basis of Lemma \ref{lemma:upper_semi_continuity_abstract}. Firstly, note that $f(x,p,\theta)=\Lambda(x,p,\theta)-\cI(x,p,\theta)$ is upper semi-continuous since, by Assumption \ref{assumption:regularity:Lambda-I} \ref{item:assumption:Gamma-convergence}, $-f(x,p,\theta)= \cI(x,p,\theta)-\Lambda(x,p,\theta)$ is lower semi-continuous.
Moreover, by Assumption \ref{assumption:regularity:Lambda-I} \ref{item:assumption:statinary_meas}, for all small neighborhood of $x$ and $p$ there exists a continuous function $g$ in these neighborhoods such that there exists $\theta^0(x,p)$ for which $\cI(x,p,\theta^0(x,p)) - \Lambda(x,p,\theta^0(x,p)) \leq g(x,p)$. Hence, we can write the supremum over $\theta \in \Theta$ as the supremum over $\theta \in \phi(x,p)$ where 
\begin{equation}
    \phi_g(x,p)=\overline{\left\{ \theta \in \Theta \, \vert \, \cI(x,p,\theta)-\Lambda(x,p,\theta)\leq g(x,p) \right\}}.
\end{equation}
 $\phi_g(x,p)$ is non empty, since $\theta^0(x,p) \in \phi_g(x,p)$, and it is compact for Assumption \ref{assumption:regularity:Lambda-I} \ref{item:assumption:statinary_meas}. We are left to show that $\phi$ is upper hemi-continuous. 
Thus, let $(x_n,p_n) \rightarrow (x,p)$ and $\theta_n \rightarrow \theta$ with $\theta_n \in \phi_g(x_n,p_n)$. We establish that $\theta \in \phi_g(x,p)$. By the definition of $\phi_g(x_n,p_n)$, $\cI(x_n,p_n,\theta_n)- \Lambda(x_n, p_n, \theta_n)\leq g(x,p)$. Then, by the lower semi-continuity of $\cI-\Lambda$, we can write
 \begin{align}
\cI(x,p,\theta) - \Lambda(x,p,\theta) &\leq \liminf_n \cI(x_n,p_n,\theta_n) - \Lambda(x_n,p_n,\theta_n) \leq \liminf_{n} g(x_n,p_n)=g(x,p),	
\end{align}
which implies indeed that $\theta  \in \phi_g(x,p)$. Thus, upper semi-continuity follows by an application of Lemma \ref{lemma:upper_semi_continuity_abstract}.

We prove now the lower semi--continuity of $\cH$. Precisely, we want to show that if $(x_n,p_n)\rightarrow (x,p)$ then $\liminf_n \cH(x_n,p_n)\geq \cH(x,p)$. Let $\theta\in\Theta$ be such that $\cH(x,p)=\Lambda(x,p,\theta)-\cI(x,p,\theta)$.
By Assumption \ref{assumption:regularity:Lambda-I}\ref{item:assumption:Gamma-convergence}, the function $\cI-\Lambda$ is $\Gamma$--convergent in the sense of Definition \ref{def:Gamma-convergence}. That means that there exist $\theta_n$ converging to $\theta$ such that $\limsup_n \cI(x_n,p_n,\theta_n)- \Lambda(x_n,p_n,\theta_n) \leq \cI(x,p,\theta)-\Lambda(x,p,\theta)$. Therefore, 
\begin{align}
    \liminf_n \cH(x_n,p_n)&\geq \liminf_n \left[\Lambda(x_n,p_n,\theta_n) - \cI(x_n,p_n,\theta_n)\right]\\
    &= -\limsup_n \left[\cI(x_n,p_n,\theta_n)-\Lambda(x_n,p_n,\theta_n) \right]\geq - \left[\cI(x,p,\theta)- \Lambda(x,p,\theta) \right]\\
    &=\cH(x,p),
\end{align}
establishing the lower semi-continuity of $\cH$ and, hence, the continuity.
Moreover, since the Lagrangian $\cL$ is the Legendre transform of $\cH$, it is lower semi-continuous.
\end{proof}

\section{Proofs of the main theorems} \label{section:comparison_principle}
In this section we establish Theorem \ref{theorem:comparison_principle_variational} and Theorem \ref{theorem:existence_of_viscosity_solution}. 

To prove the comparison principle for $f-\lambda \bfH f= h$ and $\partial_t f - \bfH f = 0$, we relate them to a set of Hamilton--Jacobi--Bellman equation with Hamiltonians constructed from $\bfH$. To do this, we introduce two operators $H_\dagger$ and $H_\ddagger$ that will be respectively an upper and lower bound for $\bfH$. The two new Hamiltonians, defined in Subsection \ref{subsection:definition_of_Hamiltonians}, are constructed in terms of a containment function $\Upsilon$ that allows us to restrict our analysis to a compact set.
 Schematically, we will establish the diagram in Figure \ref{SF:fig:CP-diagram-in-proof-of-CP}.
 
 \begin{figure}
\begin{center}
	\begin{tikzpicture}
		\matrix (m) [matrix of math nodes,row sep=1em,column sep=4em,minimum width=2em]
		{
			{ } &[7mm] H_\dagger \\
			\bfH & { } \\
			{ }  & H_\ddagger \\};
		\path[-stealth]
		(m-2-1) edge node [above] {sub \qquad { }} (m-1-2)
		(m-2-1) edge node [below] {super \qquad { }} (m-3-2);
		
		\begin{pgfonlayer}{background}
			\node at (m-2-2) [rectangle,draw=blue!50,fill=blue!20,rounded corners, minimum width=1cm, minimum height=2.5cm]  {comparison};
		\end{pgfonlayer}
	\end{tikzpicture}
  \end{center}
\captionsetup{width=.9\textwidth}
\caption{\label{SF:fig:CP-diagram-in-proof-of-CP}
In this diagram, an arrow connecting an operator $A$ with operator $B$ with subscript 'sub' means that viscosity subsolutions of $f - \lambda A f = h$ (or $\partial_t f - A f = 0$) are also viscosity subsolutions of $f - \lambda B f = h$ (or $\partial_t f -  Bf = 0)$. Similarly for arrows with a subscript 'super'.
}
\end{figure}
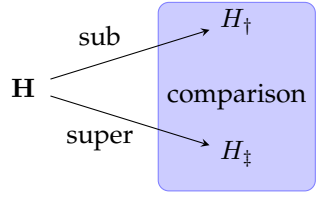

\smallskip
The arrows will be established in Subsection \ref{subsection:definition_of_Hamiltonians}. Finally, we will establish the comparison principle for $H_\dagger$ and $H_\ddagger$ in Subsection~\ref{subsection:proof_of_comparison_principle}. The combination of these two results imply the comparison principle for $\bfH$ as shown in the following.
\begin{proof}[Proof of Theorem~\ref{theorem:comparison_principle_variational}]
We only prove the first item. The proof for the time-dependent case follows the same lines.
 Fix $h_1,h_2 \in C_b(E)$ and $\lambda > 0$.
	Let $u,v$ be a viscosity sub- and supersolution to $f - \lambda \bfH f = h_1$ and  $f - \lambda \bfH f = h_2$ respectively. By Lemma \ref{lemma:viscosity_solutions_compactify2} proven in Section~\ref{subsection:definition_of_Hamiltonians}, $u$ and $v$ are a sub- and supersolution to $f - \lambda H_\dagger f = h_1$ and $f - \lambda H_\ddagger f = h_2$ respectively. Thus $\sup_E u - v \leq \sup_E h_1 - h_2$ by Proposition~\ref{prop:CP} of Section~\ref{subsection:proof_of_comparison_principle}.
\end{proof}

\subsection{Auxiliary operators} \label{subsection:definition_of_Hamiltonians}
In this section, we repeat the definition of $\bfH$, and introduce the operators $H_\dagger$ and $H_\ddagger$.
\begin{definition} \label{definition_effectiveH}
	The operator $\bfH \subseteq C_b^1(E) \times C_b(E)$ has domain $\cD(\bfH) = C_{cc}^\infty(E)$ and satisfies $\bfH f(x) = \cH(x, \nabla f(x))$, where $\cH$ is the map
	\begin{equation*}
		\mathcal{H}(x,p) = \sup_{\theta \in \Theta}\left[\Lambda(x,p,\theta) - \cI(x,p,\theta)\right].
	\end{equation*}
\end{definition}
We proceed by introducing $H_\dagger$ and $H_\ddagger$ serving as natural upper and lower bounds to $\bfH$. Recall  Assumption~\ref{item:assumption:compact_containment} and the constant $C_\Upsilon := \sup_{\theta}\sup_x \Lambda(x,\nabla \Upsilon(x),\theta)- \cI(x,\nabla \Upsilon(x),\theta)$ therein.
Denote by $C_\ell^\infty(E)$ the set of smooth functions on $E$ that have a lower bound and by $C_u^\infty(E)$ the set of smooth functions on $E$ that have an upper bound. The definitions below are motivated by the convexity of the map $p \mapsto \cH(x,p)$.
\begin{definition}[The operators $H_\dagger$ and $H_\ddagger$] \label{definiton:HdaggerHddagger}
	For $f \in C_\ell^\infty(E)$ and $\varepsilon \in (0,1)$  set 
	\begin{gather*}
		f^\varepsilon_\dagger := (1-\varepsilon) f + \varepsilon \Upsilon \\
		H_{\dagger,f}^\varepsilon(x) := (1-\varepsilon) \cH(x,\nabla f(x)) + \varepsilon C_\Upsilon.
	\end{gather*}
	and set
	\begin{equation*}
		H_\dagger := \left\{(f^\varepsilon_\dagger,H_{\dagger,f}^\varepsilon) \, \middle| \, f \in C_\ell^\infty(E), \varepsilon \in (0,1) \right\}.
	\end{equation*} 
	For $f \in C_u^\infty(E)$ and $\varepsilon \in (0,1)$  set 
	\begin{gather*}
		f^\varepsilon_\ddagger := (1+\varepsilon) f - \varepsilon \Upsilon \\
		H_{\ddagger,f}^\varepsilon(x) := (1+\varepsilon) \cH(x,\nabla f(x)) - \varepsilon C_\Upsilon.
	\end{gather*}
	and set
	\begin{equation*}
		H_\ddagger := \left\{(f^\varepsilon_\ddagger,H_{\ddagger,f}^\varepsilon) \, \middle| \, f \in C_u^\infty(E), \varepsilon \in (0,1) \right\}.
	\end{equation*} 
\end{definition}
\smallskip

The operator $\bfH$ is related to $H_\dagger, H_\ddagger$ by the following Lemma whose proof is standard and can be found for example in \cite{KrSc21}. We include it for completeness.  
	\begin{lemma}\label{lemma:viscosity_solutions_compactify2}
		Fix $\lambda > 0$ and $h \in C_b(E)$. 
		\begin{enumerate}[(a)]
			\item Every subsolution to $f - \lambda \bfH f = h$ is also a subsolution to $f - \lambda H_\dagger f = h$.
			\item Every supersolution to $f - \lambda \bfH f = h$ is also a supersolution to~$f-\lambda H_\ddagger f=~h$.
                \item Every subsolution to $\partial_t f - \bfH f = 0$ is also a subsolution to $\partial_t f - H_\dagger f = 0$.
                \item Every supersolution to $\partial_t f - \bfH f = 0$ is also a supersolution to~$\partial_t f - H_\ddagger f=0$.
		\end{enumerate}
	\end{lemma}
	\begin{proof}
 We only prove (a) as the other claims can be carried out analogously.
		Fix $\lambda > 0$ and $h \in C_b(E)$. Let $u$ be a subsolution to $f - \lambda \mathbf{H}f = h$. We prove it is also a subsolution to $f - \lambda H_\dagger f = h$.
		\smallskip
		
		Fix $\varepsilon > 0 $ and $f\in C_\ell^\infty(E)$ and let $(f^\varepsilon_\dagger,H^\varepsilon_{\dagger,f}) \in H_\dagger$ as in Definition \ref{definiton:HdaggerHddagger}. We will prove that there are $x_n\in E$ such that
		\begin{gather}
			\lim_{n\to\infty}\left(u-f_\dagger^\varepsilon\right)(x_n) = \sup_{x\in E}\left(u(x)-f_\dagger^\varepsilon(x) \right),\label{eqn:proof_lemma_conditions_for_subsolution_first}\\
			\limsup_{n\to\infty} \left[u(x_n)-\lambda H_{\dagger,f}^\varepsilon(x_n) - h(x_n)\right]\leq 0.\label{eqn:proof_lemma_conditions_for_subsolution_second}
		\end{gather}
		As the function $\left[u -(1-\varepsilon)f\right]$ is bounded from above and $\varepsilon \Upsilon$ has compact sublevel-sets, the sequence $x_n$ along which the first limit is attained can be assumed to lie in the compact set 
		\begin{equation*}
			K := \left\{x \, | \, \Upsilon(x) \leq \varepsilon^{-1} \sup_x \left(u(x) - (1-\varepsilon)f(x) \right)\right\}.
		\end{equation*}
		Set $M = \varepsilon^{-1} \sup_x \left(u(x) - (1-\varepsilon)f(x) \right)$. Let $\gamma : \bR \rightarrow \bR$ be a smooth increasing function such that
		\begin{equation*}
			\gamma(r) = \begin{cases}
				r & \text{if } r \leq M, \\
				M + 1 & \text{if } r \geq M+2.
			\end{cases}
		\end{equation*}
		Denote by $f_\varepsilon$ the function on $E$ defined by 
		\begin{equation*}
			f_\varepsilon(x) := \gamma\left((1-\varepsilon)f(x) + \varepsilon \Upsilon(x) \right).
		\end{equation*}
		By construction $f_\varepsilon$ is smooth and constant outside of a compact set and thus lies in $\cD(H) = C_{cc}^\infty(E)$. As $u$ is a viscosity subsolution for $f - \lambda Hf = h$ there exists a sequence $x_n \in K \subseteq E$ (by our choice of $K$) with
		\begin{gather}
			\lim_n \left(u-f_\varepsilon\right)(x_n) = \sup_x \left(u(x)-f_\varepsilon(x)\right), \label{eqn:visc_subsol_sup} \\
			\limsup_n \left[u(x_n) - \lambda \mathbf{H} f_\varepsilon(x_n) - h(x_n)\right] \leq 0. \label{eqn:visc_subsol_upperbound}
		\end{gather}
		As $f_\varepsilon$ equals $f_\dagger^\varepsilon$ on $K$, we have from \eqref{eqn:visc_subsol_sup} that also
		\begin{equation*}
			\lim_n \left(u-f_\dagger^\varepsilon\right)(x_n) = \sup_{x\in E}\left(u(x)-f_\dagger^\varepsilon(x)\right),
		\end{equation*}
		establishing~\eqref{eqn:proof_lemma_conditions_for_subsolution_first}. Convexity of $p \mapsto \mathcal{H}(x,p)$ yields for arbitrary points $x\in K$ the estimate
		\begin{align*}
			\mathbf{H} f_\varepsilon(x) &= \mathcal{H}(x,\nabla f_\varepsilon(x)) \\
			& \leq (1-\varepsilon) \mathcal{H}(x,\nabla f(x)) + \varepsilon \mathcal{H}(x,\nabla \Upsilon(x)) \\
			&\leq (1-\varepsilon) \mathcal{H}(x,\nabla f(x)) + \varepsilon C_\Upsilon = H^\varepsilon_{\dagger,f}(x).
		\end{align*} 
		Combining this inequality with \eqref{eqn:visc_subsol_upperbound} yields
		\begin{equation}
			\limsup_n \left[u(x_n) - \lambda H^\varepsilon_{\dagger,f}(x_n) - h(x_n)\right] 
			\leq \limsup_n \left[u(x_n) - \lambda \mathbf{H} f_\varepsilon(x_n) - h(x_n)\right] \leq 0,
		\end{equation}
		establishing \eqref{eqn:proof_lemma_conditions_for_subsolution_second}. This concludes the proof.
	\end{proof}

\subsection{The comparison principle} \label{subsection:proof_of_comparison_principle}
In the following we prove the comparison principle for the operators $H_\dagger$ and $H_\ddagger$.
\begin{proposition}\label{prop:CP} 
	Fix $\lambda > 0$ and $h_1,h_2 \in C_b(E)$. The following holds:
 \begin{enumerate}[(a)]
     \item Let $u$ be a viscosity subsolution to $f - \lambda H_\dagger f = h_1$ and let $v$ be a viscosity supersolution to $f - \lambda H_\ddagger f = h_2$. Then we have $\sup_x u(x) - v(x) \leq \sup_x h_1(x) - h_2(x)$.
     \item Let $u$ be a viscosity subsolution to $\partial_t f -  H_\dagger f = 0$ and let $v$ be a viscosity supersolution to $\partial_t f - H_\ddagger f = 0$. Then we have $\sup_{t\in [0,T], x} u(x,t) - v(x,t) \leq \sup_x{ u(x,0) - v(x,0)}$ for all $T>0$.
 \end{enumerate}
\end{proposition}

The strategy of the proof is the same for both equations. 
In both cases, the aim is to prove that it is possible to bound the difference of the Hamiltonians, in well-chosen sequences of points, by using the continuity estimate for $\Lambda - \cI$. This is the content of Proposition \ref{prop:continuity_estimate}. With this aim, we use two variants of a classical estimate, that was proven e.g. in \cite[Proposition 3.7]{CIL92}, given respectively in Lemma \ref{lemma:doubling_lemma} for the stationary equation and Lemma \ref{lemma:quadrupling_lemma} for the evolutionary case.

\begin{lemma}\label{lemma:doubling_lemma}
	Let $u$ be bounded and upper semi-continuous, let $v$ be bounded and lower semi-continuous and let $\Upsilon$ be a containment function.
	\smallskip
	
	Fix $\varepsilon > 0$. For every $\alpha >0$ there exists $(x_\alpha,y_\alpha)=(x_{\alpha,\varepsilon},y_{\alpha,\varepsilon}) \in E \times E$ such that
	\begin{multline}
		\frac{u(x_{\alpha})}{1-\varepsilon} - \frac{v(y_{\alpha})}{1+\varepsilon} - \frac{\alpha}{2} d^2(x_{\alpha},y_{\alpha}) - \frac{\varepsilon}{1-\varepsilon}\Upsilon(x_{\alpha}) -\frac{\varepsilon}{1+\varepsilon}\Upsilon(y_{\alpha}) \\
		= \sup_{x,y \in E} \left\{\frac{u(x)}{1-\varepsilon} - \frac{v(y)}{1+\varepsilon} -  \frac{\alpha }{2}d^2(x,y)  - \frac{\varepsilon}{1-\varepsilon}\Upsilon(x) - \frac{\varepsilon}{1+\varepsilon}\Upsilon(y)\right\}.
	\end{multline}
	Additionally, for every $\varepsilon > 0$ we have that
	\begin{enumerate}[(a)]
		\item The set $\{x_{\alpha}, y_{\alpha} \, | \,  \alpha > 0\}$ is relatively compact in $E$.
		\item All limit points of $\{(x_{\alpha},y_{\alpha})\}_{\alpha > 0}$ as $\alpha \rightarrow \infty$ are of the form $(z,z)$ and for these limit points we have $\frac{u(z)}{1-\eps} - \frac{v(z)}{1+\eps} - \frac{2\eps}{1-\eps^2}\Upsilon(z) = \sup_{x \in E} \left\{\frac{u(x)}{1-\eps} - \frac{v(x)}{1+\eps} - \frac{2\eps}{1-\eps^2}\Upsilon(x) \right\}$.
		\item We have 
		\[
		\lim_{\alpha \rightarrow \infty}  \alpha d^2(x_{\alpha},y_{\alpha}) = 0.
		\]
	\end{enumerate}
\end{lemma}
 
\begin{lemma}\label{lemma:quadrupling_lemma}
Let $u$ be bounded and upper semi-continuous, let $v$ be bounded and lower semi-continuous and let $\Upsilon$ be a containment function.
	\smallskip
	
	Fix $\varepsilon > 0$, $\beta>0$ and $T>0$. For every $\alpha >0$ and $\gamma >0$ there exists $$(x_{\alpha,\gamma},t_{\alpha,\gamma}, y_{\alpha,\gamma}, s_{\alpha,\gamma}) = (x_{\alpha,\gamma,\eps,\beta}, t_{\alpha,\gamma,\eps,\beta}, y_{\alpha,\gamma,\eps,\beta}, s_{\alpha,\gamma,\eps,\beta})\in E\times [0,T], \times E \times [0,T]$$ such that
 \begin{multline}
\frac{u(t_{\alpha,\gamma}, x_{\alpha,\gamma})}{1-\varepsilon} - \frac{v(s_{\alpha,\gamma},y_{\alpha,\gamma})}{1+\varepsilon} - \frac{\alpha}{2}d^2(x_{\alpha,\gamma},y_{\alpha,\gamma}) - \frac{\gamma}{2}(s_{\alpha,\gamma}-t_{\alpha,\gamma})^2
    - \frac{\varepsilon}{1-\varepsilon} \Upsilon(x_{\alpha,\gamma}) - \frac{\varepsilon}{1+\varepsilon}\Upsilon(y_{\alpha,\gamma})
    \\ 
    - \frac{\beta}{2} (t_{\alpha,\gamma}+s_{\alpha,\gamma}) + \beta T \\
		= \sup_{s,t \in [0,T],x,y} \biggl\{\frac{u(t,x)}{1-\varepsilon} - \frac{v(s,y)}{1+\varepsilon} - \frac{\alpha}{2}d^2(x,y) - \frac{\alpha}{2}(s-t)^2
    - \frac{\varepsilon}{1-\varepsilon} \Upsilon(x) - \frac{\varepsilon}{1+\varepsilon}\Upsilon(y) - \frac{\beta}{2} (t+s) + \beta T \biggr\}.
 \end{multline}
	Additionally, for every $\varepsilon > 0$ and $\beta>0$ we have that
 \begin{enumerate}[(a)]
		\item For any $\gamma>0$,
  \begin{enumerate}[(i)]
      \item the set $\{x_{\alpha,\gamma}, y_{\alpha,\gamma} \, | \,  \alpha > 0\}$ is relatively compact in $E$,
      \item we have $$\lim_{\alpha \rightarrow \infty}  \alpha d^2(x_{\alpha,\gamma},y_{\alpha,\gamma}) = 0,$$
      \item all limit points of $\{(x_{\alpha,\gamma},y_{\alpha,\gamma},t_{\alpha,\gamma},s_{\alpha,\gamma})\}_{\alpha > 0}$ as $\alpha \rightarrow \infty$ are of the form $(z_{\gamma},z_{\gamma},t_\gamma,s_\gamma)$. 
      \end{enumerate}
	\item Let $(z_\gamma,z_\gamma,t_\gamma,s_\gamma)$ be a limit point as in (a)(iii). Then
 \begin{enumerate}[(i)]
     \item the set $\{z_\gamma \, | \, \gamma >0\}$ is relatively compact in $E$,
     \item all limit point of $\{(z_\gamma,z_\gamma,t_\gamma,s_\gamma)\}_{\gamma>0}$ as $\gamma \rightarrow \infty$ are of the form $(z,z,w,w)$ and for these limit points we have $$\frac{u(w,z)}{1-\eps} - \frac{v(w,z)}{1+\eps} - \frac{2\eps}{1-\eps^2}\Upsilon(z) - \beta ( T - w) = \sup_{x \in E, t\in[0,T]} \left\{\frac{u(t,x)}{1-\eps} - \frac{v(t,x)}{1+\eps} -\frac{2\eps}{1-\eps^2}\Upsilon (x) - \beta (T - t) \right\}.$$
 \end{enumerate}
	\end{enumerate}  
\end{lemma}
In the following proposition we prove the continuity estimate for $\cH$ by using the continuity estimate of $\Lambda - \cI$.
\begin{proposition}\label{prop:continuity_estimate}
    Consider $(x_{\alpha,\eps},y_{\alpha,\eps})$ found in Lemma \ref{lemma:doubling_lemma} or Lemma \ref{lemma:quadrupling_lemma} (for which we fix $\gamma$ and $\beta$) and denote $p^1_{\alpha,\eps} := \alpha \dd_x \frac{1}{2}d^2(\cdot,y_{\alpha,\eps})(x_{\alpha,\eps})$ and $p^2_{\alpha,\eps} := - \alpha \dd_y \frac{1}{2}d^2(x_{\alpha,\eps},\cdot)(y_{\alpha,\eps})$. Suppose that 
    \begin{equation}\label{subsolution_bound}
        \inf_\alpha \cH(x_{\alpha,\eps},p^1_{\alpha,\eps}) > - \infty
    \end{equation}
    and 
    \begin{equation}\label{supersolution_bound}
        \sup_\alpha \cH(y_{\alpha,\eps},p^2_{\alpha,\eps}) < \infty.
    \end{equation}
 
    Then, for all $\eps>0$ there exists  a sequence $\alpha(\eps)\to \infty$, such that 
    \begin{equation}\label{eqn:difference_hamiltonian}
        \liminf_{\eps\to 0}\liminf_{\alpha\to\infty}\cH(x_\alpha,p^1_{\alpha,\eps})- \cH(y_{\alpha,\eps},p^2_{\alpha,\eps})\leq 0.
    \end{equation}
\end{proposition}
\begin{proof}
    We only prove the statement for $(x_{\alpha,\eps},y_{\alpha,\eps})$ found in Lemma \ref{lemma:doubling_lemma}. The proof in the context of Lemma \ref{lemma:quadrupling_lemma} is analogous.
    The proof is given in two steps. We sketch the steps, before giving full proof.
    
    \underline{\emph{Step 1}}: We will show that there are controls $\theta_{\alpha,\eps}$ such that 
	\begin{equation} \label{eqn:choice_control}
		\mathcal{H}(x_{\alpha,\eps},p^1_{\alpha,\eps}) = \Lambda(x_{\alpha,\eps},p^1_{\alpha,\eps},\theta_{\alpha,\eps}) - \mathcal{I}(x_{\alpha,\eps},p^1_{\alpha,\eps},\theta_{\alpha,\eps}).
	\end{equation}
	As a consequence we have
	\begin{multline} \label{eqn:basic_decomposition_Hamiltonian_difference1}
		\mathcal{H}(x_{\alpha,\eps},p^1_{\alpha,\eps})-
		\mathcal{H}(y_{\alpha,\eps},p^2_{\alpha,\eps})
		\leq 
		\Lambda(x_{\alpha,\eps},p^1_{\alpha,\eps},\theta_{\alpha,\eps})-
		\Lambda(y_{\alpha,\eps},p^2_{\alpha,\eps},\theta_{\alpha,\eps})\\
		+\mathcal{I}(y_{\alpha,\eps},p^2_{\alpha,\eps},\theta_{\alpha,\eps})-
		\mathcal{I}(x_{\alpha,\eps},p^1_{\alpha,\eps},\theta_{\alpha,\eps}).
	\end{multline}
	For establishing \eqref{eqn:difference_hamiltonian}, it is sufficient to bound the differences in \eqref{eqn:basic_decomposition_Hamiltonian_difference1} by using Assumption \ref{assumption:regularity:Lambda-I}\ref{item:assumption:regularity:continuity_estimate}.
	
	\underline{\emph{Step 2}}: We verify the conditions to apply the continuity estimate, Assumption \ref{assumption:regularity:Lambda-I} \ref{item:assumption:regularity:continuity_estimate} which then concludes the proof.
 	\smallskip

 \underline{\emph{Proof of Step 1}}: Recall that $\mathcal{H}(x,p)$ is given by
	\begin{equation*}
		\mathcal{H}(x,p) = \sup_{\theta \in \Theta}\left[\Lambda(x,p,\theta) - \mathcal{I}(x,p,\theta) \right].
	\end{equation*}
	Since $\Lambda(x_{\alpha,\eps},p^1_{\alpha,\eps},\cdot) - \cI(x_{\alpha,\eps},p^1_{\alpha,\eps},\cdot) : \Theta \to \mathbb{R}$ is upper semi-continuous and bounded by \ref{assumption:regularity:Lambda-I}\ref{item:assumption:Gamma-convergence}  and \ref{assumption:regularity:Lambda-I}\ref{item:assumption:regularity:boundness}, there exists an optimizer $\theta_{\alpha,\eps} \in\Theta$ such that
	\begin{equation} \label{eqn:choice_of_optimal_measure}
		\mathcal{H}(x_{\alpha,\eps},p^1_{\alpha,\eps}) = \Lambda(x_{\alpha,\eps},p^1_{\alpha,\eps},\theta_{\alpha,\eps}) - \mathcal{I}(x_{\alpha,\eps},p^1_{\alpha,\eps},\theta_{\alpha,\eps})  .
	\end{equation}
	Choosing the same point in the supremum of the second term $\mathcal{H}(y_{\alpha,\eps},p^2_{\alpha,\eps})$, we obtain for all $\varepsilon > 0$ and $\alpha > 0$ the estimate
	\begin{multline} \label{eqn:basic_decomposition_Hamiltonian_difference}
		\mathcal{H}(x_{\alpha,\eps},p^1_{\alpha,\eps})-
		\mathcal{H}(y_{\alpha,\eps},p^2_{\alpha,\eps})
		\leq 
		\Lambda(x_{\alpha,\eps},p^1_{\alpha,\eps},\theta_{\alpha,\eps})-
		\Lambda(y_{\alpha,\eps},p^2_{\alpha,\eps},\theta_{\alpha,\eps})\\
		+ \mathcal{I}(y_{\alpha,\eps},p^1_{\alpha,\eps},\theta_{\alpha,\eps})-
		\mathcal{I}(x_{\alpha,\eps},p^2_{\alpha,\eps},\theta_{\alpha,\eps})  .
	\end{multline}

 \smallskip

 \underline{\emph{Proof of Step 2}}: 
	We will construct for each $\varepsilon > 0$ a sequence $\alpha = \alpha(\varepsilon) \rightarrow \infty$ such that the collection $(x_{\alpha,\eps},y_{\alpha,\eps},\theta_{\alpha,\eps})$ is fundamental for $\Lambda-\cI$ in the sense of Definition \ref{def:results:continuity_estimate}. We thus need to verify for each $\varepsilon > 0$ 
	\begin{enumerate}[(i)]
		\item \label{item:fundamental_liminf}
		\begin{equation}
			\inf_\alpha \Lambda(x_{\alpha,\eps},p^1_{\alpha,\eps},\theta_{\alpha,\eps}) - \cI(x_{\alpha,\eps},p^1_{\alpha,\eps},\theta_{\alpha,\eps}) > - \infty,\label{eq:proof-CP:Vx-unif-bound-below_first}
		\end{equation}
		\item \label{item:fundamental_limsup}
		\begin{equation}
			\sup_{\alpha}\Lambda(y_{\alpha,\eps},p^2_{\alpha,\eps},\theta_{\alpha,\eps})-\cI(y_{\alpha,\eps},p^2_{\alpha,\eps},\theta_{\alpha,\eps}) < \infty
		\end{equation}
		\item \label{item:fundamental_compactcontrols} The set of controls $\theta_{\alpha,\eps}$ is relatively compact.
	\end{enumerate}
	
	We will first establish \ref{item:fundamental_liminf} and \ref{item:fundamental_limsup} for all $\alpha$. Then, \ref{item:fundamental_compactcontrols} will follow from \ref{item:fundamental_liminf} and \ref{item:fundamental_limsup} and Assumption \ref{assumption:regularity:Lambda-I}\ref{item:assumption:compact-sublevelsets}. 
	
By \eqref{subsolution_bound} and \eqref{eqn:choice_of_optimal_measure}, 
\begin{equation}
    - \infty < \inf_\alpha \cH(x_{\alpha,\eps}, p^1_{\alpha,\eps}) = \inf_\alpha \Lambda(x_{\alpha,\eps},p^1_{\alpha,\eps},\theta_{\alpha,\eps}) - \cI (x_{\alpha,\eps},p^1_{\alpha,\eps},\theta_{\alpha,\eps})
\end{equation}
establishing \ref{item:fundamental_liminf}.

By \eqref{supersolution_bound},
\begin{equation}
    \sup_\alpha \Lambda(y_{\alpha,\eps},p^2_{\alpha,\eps},\theta_{\alpha,\eps}) - \cI (y_{\alpha,\eps},p^2_{\alpha,\eps},\theta_{\alpha,\eps}) < \sup_\alpha \cH(y_{\alpha,\eps},p^2_{\alpha,\eps}) < \infty
\end{equation}
implying \ref{item:fundamental_limsup}.

	
	\smallskip
	
\end{proof}
\begin{proof}[Proof of Proposition \ref{prop:CP}]

    \textit{Proof of (a).} Fix $\lambda >0$ and $h_1,h_2 \in C_b(E)$. Let $u$ be a viscosity subsolution and $v$ be a viscosity supersolution of $f - \lambda H_\dagger f = h_1$ and  $f - \lambda H_\ddagger f = h_2$ respectively. 
	For any $\varepsilon > 0$ and any $\alpha > 0$, define the map $\Phi_{\alpha,\eps}: E \times E \to \mathbb{R}$ by
	\begin{equation*}
		\Phi_{\alpha,\eps}(x,y) := \frac{u(x)}{1-\varepsilon} - \frac{v(y)}{1+\varepsilon} - \frac{\alpha}{2} d^2(x,y) - \frac{\varepsilon}{1-\varepsilon} \Upsilon(x) - \frac{\varepsilon}{1+\varepsilon}\Upsilon(y).
	\end{equation*}
	Let $\varepsilon > 0$. By Lemma \ref{lemma:doubling_lemma}, there is a compact set $K_\varepsilon \subseteq E$ and there exist points $x_{\alpha,\eps},y_{\alpha,\eps} \in K_\varepsilon$ such that
	\begin{equation} \label{eqn:comparison_optimizers}
		\Phi_{\alpha,\eps}(x_{\alpha,\eps},y_{\alpha,\eps}) = \sup_{x,y \in E} \Phi_{\alpha,\eps}(x,y),
	\end{equation}
	and 
	\begin{equation}\label{eq:proof-CP:Psi-xy-converge}
		\lim_{\alpha \to \infty} \frac{\alpha }{2}d^2(x_{\alpha,\eps},y_{\alpha,\eps}) = 0.
	\end{equation}
	For all $\alpha$ it follows that
	\begin{align}\label{eq:proof-CP:general-bound-u1u2}
		\sup_E (u - v) & = \lim_{\eps\to 0} \sup_{x\in E} \frac{u(x)}{1-\eps} - \frac{v(x)}{1+\eps}\\
  &\leq \liminf_{\eps \to 0} \sup_{x,y \in E} \frac{u(x)}{1-\eps} - \frac{v(y)}{1+\eps} - \frac{\alpha}{2} d^2(x,y) - \frac{\eps}{1-\eps} \Upsilon (x) - \frac{\eps}{1+\eps} \Upsilon (y) \\
  & = \liminf_{\eps \to 0} \frac{u(x_{\alpha,\varepsilon})}{1-\varepsilon} - \frac{v(y_{\alpha,\varepsilon})}{1+\varepsilon} - \frac{\alpha}{2} d^2(x_{\alpha,\varepsilon},y_{\alpha,\varepsilon}) - \frac{\varepsilon}{1-\varepsilon}\Upsilon(x_{\alpha,\varepsilon}) -\frac{\varepsilon}{1+\varepsilon}\Upsilon(y_{\alpha,\varepsilon})\\
  &\leq \liminf_{\varepsilon \to 0} \left[ \frac{u(x_{\alpha,\eps})}{1-\varepsilon} - \frac{v(y_{\alpha,\eps})}{1+\varepsilon}\right].
	\end{align}
	At this point, we want to use the sub- and supersolution properties of $u$ and $v$. Define the test functions $\varphi^{\varepsilon,\alpha}_1 \in \cD(H_\dagger), \varphi^{\varepsilon,\alpha}_2 \in \cD(H_\ddagger)$ by
	\begin{align*}
		\varphi^{\varepsilon,\alpha}_1(x) & := (1-\varepsilon) \left[\frac{v(y_{\alpha,\eps})}{1+\varepsilon} + \frac{\alpha}{2} d^2(x,y_{\alpha,\eps}) + \frac{\varepsilon}{1-\varepsilon}\Upsilon(x) + \frac{\varepsilon}{1+\varepsilon}\Upsilon(y_{\alpha,\eps})\right] \\
		& \hspace{5cm} + (1-\varepsilon)(x-x_{\alpha,\eps})^2, \\
		\varphi^{\varepsilon,\alpha}_2(y) & := (1+\varepsilon)\left[\frac{u_1(x_{\alpha,\eps})}{1-\varepsilon} - \frac{\alpha}{2} d^2(x_{\alpha,\eps},y) - \frac{\varepsilon}{1-\varepsilon}\Upsilon(x_{\alpha,\eps}) - \frac{\varepsilon}{1+\varepsilon}\Upsilon(y)\right] \\
		& \hspace{5cm}  - (1+\varepsilon) (y-y_{\alpha,\eps})^2.
	\end{align*}
	Using \eqref{eqn:comparison_optimizers}, we find that $u - \varphi^{\varepsilon,\alpha}_1$ attains its supremum at $x = x_{\alpha,\eps}$, and thus
	\begin{equation*}
		\sup_E (u-\varphi^{\varepsilon,\alpha}_1) = (u-\varphi^{\varepsilon,\alpha}_1)(x_{\alpha,\eps}).
	\end{equation*}
	Denote $p_{\alpha,\eps}^1 := \alpha \dd_x \frac{1}{2}d^2(x_{\alpha,\eps},y_{\alpha,\eps})$. By our addition of the penalization $(x-x_{\alpha,\eps})^2$ to the test function, the point $x_{\alpha,\eps}$ is in fact the unique optimizer, and we obtain from the subsolution inequality that
	\begin{equation}\label{eq:proof-CP:subsol-ineq}
		u(x_{\alpha,\eps}) - \lambda \left[ (1-\varepsilon) \mathcal{H}\left(x_{\alpha,\eps}, p_{\alpha,\eps}^1 \right) + \varepsilon C_\Upsilon\right] \leq h_1(x_{\alpha,\eps}).
	\end{equation}	
	With a similar argument for $u_2$ and $\varphi^{\varepsilon,\alpha}_2$, we obtain by the supersolution inequality that
	\begin{equation}\label{eq:proof-CP:supersol-ineq}
		v(y_{\alpha,\eps}) - \lambda \left[(1+\varepsilon)\mathcal{H}\left(y_{\alpha,\eps}, p_{\alpha,\eps}^2 \right) - \varepsilon C_\Upsilon\right] \geq h_2(y_{\alpha,\eps}),
	\end{equation}
	where $p_{\alpha,\eps}^2 := -\alpha \dd_y \frac{1}{2} d^2 (x_{\alpha,\eps},y_{\alpha,\eps})$. With that, estimating further in~\eqref{eq:proof-CP:general-bound-u1u2} leads to
	
 \begin{multline*}
		\sup_E(u-v) \leq \liminf_{\varepsilon\to 0}\liminf_{\alpha \to \infty} \bigg[\frac{h_1(x_{\alpha,\eps})}{1-\varepsilon} - \frac{h_2(y_{\alpha,\eps})}{1+\varepsilon} + \frac{\varepsilon}{1-\varepsilon} C_\Upsilon \\ + \frac{\varepsilon}{1+\varepsilon} C_\Upsilon  + \lambda \left[\mathcal{H}(x_{\alpha,\eps},p^1_{\alpha,\eps}) - \mathcal{H}(y_{\alpha,\eps},p^2_{\alpha,\eps})\right]\bigg].
	\end{multline*}
 Note that, by the subsolution inequality \eqref{eq:proof-CP:subsol-ineq},
	\begin{align}
		- \infty < \frac{1}{\lambda} \inf_E\left(u - h_1\right) & \leq (1-\varepsilon) \mathcal{H}(x_{\alpha,\eps},p^1_{\alpha,\eps}) + \varepsilon C_{\Upsilon},
	\end{align}
 and by the supersolution inequality 
 \eqref{eq:proof-CP:supersol-ineq},

\begin{equation} 
(1+\eps)\mathcal{H}\left(y_{\alpha,\eps},p^2_{\alpha,\eps}\right) - \varepsilon C_\Upsilon \leq \frac{1}{\lambda} \sup_E (v-h_2) < \infty.
\end{equation}
Thus, comparison principle follows from Proposition \ref{prop:continuity_estimate}.

 \smallskip
 
 \textit{Proof of (b).} Let $u$ be a subsolution for $H_\dagger$, and $v$ a supersolution for $H_\ddagger$. Let $T > 0$ be fixed. 


For any $\beta > 0$, we have
\begin{equation*}
    \sup_{t \in [0,T],x} u(t,x) - v(t,x) \leq \sup_{t \in [0,T],x} u(t,x) - v(t,x) -\beta t + \beta T
\end{equation*}
We next incorporate our Lyapunov type functions 
\begin{equation*}
    \sup_{t \in [0,T],x} u(t,x) - v(t,x) -\beta t + \beta T = \lim_{\varepsilon \downarrow 0} \sup_{t \in [0,T],x} \frac{u(t,x)}{1-\eps} - \frac{v(t,x)}{1+\eps} - \frac{2\eps}{1-\eps^2}\Upsilon(x) -\beta t + \beta T.
\end{equation*}
Thus, for any $\varepsilon > 0$, $\alpha,\gamma > 0$, we have
\begin{multline}\label{eq:quadruplication}
    \sup_{t \in [0,T], x} \frac{u(t,x)}{1-\varepsilon} - \frac{v(t,x)}{1+\varepsilon} - \frac{2\epsilon}{1-\epsilon^2}\Upsilon(x) -\beta t + \beta T \leq \sup_{s,t \in [0,T],x,y} \frac{u(t,x)}{1-\varepsilon} - \frac{v(s,y)}{1+\varepsilon} - \frac{\alpha}{2}d^2(x,y) - \frac{\gamma}{2}(s-t)^2 \\
    - \frac{\varepsilon}{1-\varepsilon} \Upsilon(x) - \frac{\varepsilon}{1+\varepsilon}\Upsilon(y) - \frac{\beta}{2} (t+s) + \beta T.
\end{multline}
By Lemma \ref{lemma:quadrupling_lemma}, there are $(x_{\alpha,\gamma},t_{\alpha,\gamma},y_{\alpha,\gamma},s_{\alpha,\gamma}) = (x_{\varepsilon,\beta,\alpha,\gamma},t_{\varepsilon,\beta,\alpha,\gamma},y_{\varepsilon,\beta,\alpha,\gamma},s_{\varepsilon,\beta,\alpha,\gamma})$ optimizing the supremum on the right-hand side and such that
\begin{align*}
    \frac{\alpha}{2}d^2(x_{\alpha,\gamma},y_{\alpha,\gamma}) \to 0 \quad \text{as $\alpha\to\infty$}.
\end{align*}

First assume $t_{\alpha,\gamma},s_{\alpha,\gamma} > 0$. We will aim for a contradiction.

Define the functions $f^\dagger_{\eps,\beta,\alpha}(x)\in D(H_\dagger)$ and $f^\ddagger_{\eps,\beta,\alpha}(y)\in D(H_\ddagger)$ by
\begin{align*}
    f^\dagger(x) & = (1-\eps)\left[\frac{v(s_{\alpha,\gamma},y_{\alpha,\gamma})}{1+\eps} +\frac{\alpha}{2}d^2(x,y_{\alpha,\gamma}) +\frac{\eps}{1+\eps}\Upsilon(y_{\alpha,\gamma})- \beta T  \right] +\eps \Upsilon(x),\\
f^\ddagger(y) & = (1+\eps)\left[ \frac{u(t_{\alpha,\gamma},x_{\alpha,\gamma})}{1-\eps} -\frac{\alpha}{2}d^2(x_{\alpha,\gamma},y) - \frac{\eps}{1-\eps}\Upsilon(x_{\alpha,\gamma}) +\beta T\right] - \eps \Upsilon(y).
\end{align*}
Observe that
\begin{align*}
  u(t_{\alpha,\gamma},x_{\alpha,\gamma})-f^\dagger(x_{\alpha,\gamma}) - h_1(t_{\alpha,\gamma}) & = \sup_{t\in[0,T],x} u(t,x)- f^\dagger(x)- h_1(t),\\
  v(s_{\alpha,\gamma},y_{\alpha,\gamma})-f^\ddagger(y_{\alpha,\gamma})- h_2(s_{\alpha,\gamma}) & = \inf_{t\in [0,T],y} v(s,y) - f^\ddagger(y) - h_2(s),
\end{align*}
where $h_1(t)= (1-\eps)(\frac{\beta}{2}(t+s_{\alpha,\gamma}) + \frac{\gamma}{2}(t-s_{\alpha,\gamma})^2)$ and $h_2(s)= (1+\eps)(-\frac{\beta}{2}(t_{\alpha,\gamma}+s) - \frac{\gamma}{2}(t_{\alpha,\gamma}- s)^2)$.

Then, using the sub and supersolution inequalities, dividing them by $(1-\varepsilon)$ and $(1+\varepsilon)$ respectively, we find 
\begin{align}
    & \gamma (t_{\alpha,\gamma} - s_{\alpha,\gamma}) + \frac{\beta}{2} - \cH(x_{\alpha,\gamma}, \alpha \dd_x \frac{1}{2} d^2(\cdot,y_{\alpha,\gamma})(x_{\alpha,\gamma})) - \frac{\varepsilon}{1-\varepsilon}c_\Upsilon \leq 0, \\
    & \gamma (t_{\alpha,\gamma} - s_{\alpha,\gamma}) - \frac{\beta}{2} - \cH(y_{\alpha,\gamma}, - \alpha \dd_y \frac{1}{2} d^2(x_{\alpha,\gamma},\cdot)(y_{\alpha,\gamma})) + \frac{\varepsilon}{1+\varepsilon}c_\Upsilon \geq 0.
\end{align}
Combining the two equations yields
\begin{equation*}
    \beta \leq \cH(x_{\alpha,\gamma}, \alpha \dd_x \frac{1}{2} d^2(\cdot,y_{\alpha,\gamma})(x_{\alpha,\gamma})) - \cH(y_{\alpha,\gamma}, - \alpha \dd_y \frac{1}{2} d^2(x_{\alpha,\gamma},\cdot)(y_{\alpha,\gamma})) + \frac{2\varepsilon}{1-\varepsilon^2} c_\Upsilon 
\end{equation*} 
sending $\alpha \rightarrow \infty$ and $\varepsilon \rightarrow 0$, and using Proposition \ref{prop:continuity_estimate}, we get a contradiction for small $\eps$ as $\beta > 0$.
So it holds that for small $\varepsilon$, large $\alpha$ and all $\gamma>0$, we have $t_{\alpha,\gamma} = 0$ or $s_{\alpha,\gamma} = 0$.

Proceeding from equation \eqref{eq:quadruplication}, we get
\begin{align*}
    \sup_{t \in [0,T],x} u(t,x) - v(t,x) & \leq \sup_{t \in [0,T],x} u(t,x) - v(t,x) -\beta t + \beta T \\
    & \leq  \lim_{\varepsilon \downarrow 0} \sup_{s,t \in [0,T],x,y} \frac{u(t,x)}{1-\varepsilon} - \frac{v(s,y)}{1+\varepsilon} - \frac{\alpha}{2}d^2(x,y) - \frac{\gamma}{2}(s-t)^2 \\
    & \qquad - \frac{\varepsilon}{1-\varepsilon} \Upsilon(x) - \frac{\varepsilon}{1+\varepsilon}\Upsilon(y) - \frac{\beta}{2} (t+s) + \beta T \\
    & \leq  \lim_{\varepsilon \downarrow 0} \lim_{\gamma \rightarrow \infty} \lim_{\alpha \rightarrow \infty} \frac{u(t_{\alpha,\gamma},x_{\alpha,\gamma})}{1-\varepsilon} - \frac{v(s_{\alpha,\gamma},y_{\alpha,\gamma})}{1+\varepsilon}\\
    &\qquad +\frac{\varepsilon}{1-\varepsilon} \Upsilon(x_{\alpha,\gamma}) - \frac{\varepsilon}{1+\varepsilon}\Upsilon(y_{\alpha,\gamma}) - \frac{\beta}{2} (t_{\alpha,\gamma}+s_{\alpha,\gamma}) + \beta T \\
    & \leq \sup_x u(0,x) - v(0,x) + \beta T
\end{align*}
where we used that $u$ is upper semi-continuous, $v$ is lower semi-continuous, and  Lemma \ref{lemma:quadrupling_lemma} (b)(ii) and the fact that $t_{\alpha,\gamma}=0$ or $s_{\alpha,\gamma} = 0$.
As $\beta > 0$ was arbitrary, we conclude.

\end{proof}

\subsection{Existence of viscosity solutions}\label{sec:existence}
In this subsection, we will prove Theorem \ref{theorem:existence_of_viscosity_solution}. In other words, we show that for~$h\in C_b(E)$ and~$\lambda>0$, the function~$R(\lambda)h$ given by
	\begin{equation*}
		R(\lambda) h(x) = \sup_{\substack{\gamma \in \mathcal{A}\mathcal{C} \\ \gamma(0) = x}} \int_0^\infty \lambda^{-1} e^{-\lambda^{-1}t} \left[h(\gamma(t)) - \int_0^t \mathcal{L}(\gamma(s),\dot{\gamma}(s))\right] \, \dd t
	\end{equation*}
	is indeed a viscosity solution to $f - \lambda \bfH f = h$. 
 \begin{proof}[Proof of Theorem \ref{theorem:existence_of_viscosity_solution}]
     The result follows from Theorem 2.8 in \cite{KrSchl20} where the authors show that $\partial_p \cH(x,p) \cap T_E(x) \neq \emptyset$ implies the existence of solutions.
 \end{proof}

\smallskip

We want to mention that in our ongoing work \cite{DeCoKr23} we prove the existence of solutions for the evolutionary equation as well.

\section{Example: Hamiltonians arising from biochemistry}
\label{section:verification-for-examples-of-Hamiltonians}
The purpose of this section is to showcase that the method introduced in this paper is versatile enough to capture interesting examples that could not be treated before.
The Hamiltonian that we consider arises from the study of systems biology. More precisely, it plays a crucial role in the study, with a \textit{large deviations} approach, of multi--scale Markov processes modelling chemical reactions.
Over the past few decades, there has been significant research conducted on multi-scale chemical reactions networks (see for example \cite{KaKu13} and \cite{BaKuPoRe06}).
They are usually described by the use of continuous-time Markov chains with generators of the following form 
\begin{equation}\label{eq:multiple-time-Markov-process-generator}
     A f(z) =  \sum_{\gamma \in \Gamma} r(z,\gamma)\left[f(z+\gamma) - f(z) \right],
\end{equation}
with $z$ in the space $\mathbb{N}^J$, with $J$ the set of chemical species, and $r\in C^1 (\mathbb{N}^J\times \Gamma)$ is a non-negative smooth function.
In the above, the state $z$ is a vector whose components  describe the number of molecules of a chemical species,  $r$ is the transition rate of the reaction, $\Gamma$ is the set representing all reactions. In particular, every $\gamma \in \Gamma$ describes a reaction in the sense that its $i$--th component represents the number of molecules of the $i$--th species that is used (if the component is negative) or obtained (if the component is positive) in the reaction.

Motivated by a huge class of examples arising from biochemistry in which two dominant time--scales occur (see for instance Example \ref{ex:mm}), we will consider a two - scale process $Z=(X,Y)$ in the space $E^0 = \mathbb{N}^l \times \mathbb{N}^m$.  
The amount of molecules of the first type is an order of magnitude greater then the amount of the second type.  For this reason and to be able to study the limit behaviour of the process, we consider the scaled species $X_N=X/N$ and $Y_N=Y$ in the space $E^0_N = \left(\frac{1}{N} \mathbb{N}\right)^l \times \mathbb{N}^m$. The time-scale separation between the slow process $X_N$ and the fast process $Y_N$ is then $N$ and the generator of the rescaled process $Z_N=(X_N,Y_N)$ is given by
\begin{equation}\label{eq:generator_rescaled}
    A_N f(x,y) = N \sum_{\gamma=(\gamma_x,\gamma_y) \in \Gamma} r(x,y,\gamma) [f(x+N^{-1}\gamma_x , y + \gamma_y)],
\end{equation}
for $f \in D(A_N) \subseteq C (E_N^0)$.

In order to gain a more comprehensive understanding of the system and simplify the limit procedure that will follow, we divide the generator into three distinct parts. These parts will specifically describe reactions occurring on the macroscopic, microscopic, and a combination of the two scales, respectively.
The generator of $(X_N, Y_N)$ is then, 
\begin{align}\label{eq:multiple-time: generator_extended}
    A_N f(x,y) = & N \sum_{\gamma=(\gamma_x,\gamma_y)\in \Gamma_1} r (x,y,\gamma)\left[f(x+N^{-1}\gamma_x,y) - f(x,y) \right] + \\
    & N \sum_{\gamma=(\gamma_x,\gamma_y)\in \Gamma_2} r (x,y,\gamma)\left[f(x+N^{-1}\gamma_x,y + \gamma_y) - f(x,y) \right] + \\
    & N \sum_{\gamma=(\gamma_x,\gamma_y)\in \Gamma_3} r (x,y,\gamma)\left[f(x,y + \gamma_y) - f(x,y) \right],
\end{align}
where we write
\begin{align}
    &\Gamma_1 = \left\{ \gamma=(\gamma_x,\gamma_y)\in \bZ^l \times \bZ^m \, : \, \gamma_{y_i}=0 \,  \forall i \in \{1,\dots,m\} \right\} ;\\
    &\Gamma_2 = \left\{ \gamma =(\gamma_x,\gamma_y)\in \bZ^l \times \bZ^m \, : \, \exists i\in\{1,\dots,l\} \,,  j \in \{1,\dots, m\} \, | \, \gamma_{x_i}\neq 0, \gamma_{y_j}\neq 0 \right\} ;\\
    &\Gamma_3 = \left\{ \gamma =(\gamma_x,\gamma_y)\in \bZ^l \times \bZ^m \, : \, \gamma_{x_i}=0 \,  \forall i \in \{1,\dots,l\} \right\}.
\end{align}
The model is subjected to the following assumption.
\begin{assumption}\label{assumption:conservation}
	The molecules that are part of the fast process are subjected to a conservation law. More precisely, there exists a constant $M>0$ such that
	\begin{equation}
		\sum_{i=1}^m Y_{i} = M, \qquad \text{and} \qquad \sum_{i=1}^m \gamma_{y,i} = 0 \qquad \forall \gamma \in \Gamma_2 \cup \Gamma_3.
	\end{equation}
\end{assumption}

Assumption \ref{assumption:conservation} allows us to restrict the set of values of $Y_N$ to the set 
 $$F_M = \{ n \in \mathbb{N}^m \, : \, \sum_{i= 1}^m n_i = M \},$$
 and, hence, we consider for our analysis of $Z_N$ the set
 $$E_N = \left(\frac{1}{N} \mathbb{N}\right)^l \times F_M.$$

To show an example, we describe in the following the model for enzyme kinetics with an inflow of the substrate, also called Michaelis--Menten model, studied in \cite{Po18}.
\begin{example}\label{ex:mm}
Consider four types of molecules. Namely, $S$, $E$, $ES$ and $P$ representing respectively the substrate, enzyme, enzyme--substrate complex and the product. The following four reactions occur.
\begin{enumerate}[(1)]
    \item $\emptyset \xhookrightarrow[]{k_0} S$
    \item $E + S \xhookrightarrow[]{k_1} ES$  \item $ES \xhookrightarrow[]{k_2} E + S$
    \item $ES \xhookrightarrow[]{k_3} P +E$.
\end{enumerate}
Let $X^1, X^2, Y^1, Y^2$ represent the amount of $S, P, E, ES$ respectively. 

In real-world physical scenarios, the quantities of enzyme molecules and enzyme-substrate complexes are typically small when compared to the number of substrate and product molecules. Consequently, it is reasonable to assume that $X^1$ and $X^2$ are an order of magnitude greater then $Y^1$ and $Y^2$. In this way, we lead to the scaled amounts represented by a slow process $X_N = (X^1/ N, X^2/N)$ and a fast process $Y_N=(Y^1, Y^2)$. 

The generator of the two--scales process $Z_N = (X_N,Y_N)$ is as in \eqref{eq:multiple-time: generator_extended}, with the following rates, each describing one of the reactions above.
\begin{enumerate}[(1)]
    \item $r(x,y,(1,0, 0, 0))= k_0 \in \R_+;$
    \item $r(x,y,(-1, 0 , -1, 1)) = k_1 x_1 y_1  \qquad k_1, \in \R_+;$
    \item $r(x,y,(1,0,1,-1)) = k_2 y_2 \qquad k_2 \in \R_+;$
    \item $r(x,y,(0,1,1,-1)) = k_3 y_2 \qquad k_3\in \R_+$.
\end{enumerate}
\end{example}
The above model falls into the class of examples described by \eqref{eq:multiple-time: generator_extended}. Moreover, we will observe in Subsection \ref{subsection:example_Hamiltonian} that our approach proves to be suitable and applicable to a bigger class of scenarios.  

We are interested in the limit behaviour of the two component $E_N$-valued Markov process $(X_N, Y_N)$. In the limit regime, the fast component $Y_N$ converges to equilibrium and the slow component $X_N$ converges to a deterministic limit. To characterize the speed of convergence, we are interested in the large deviation behavior for the slow process. In particular, we want to characterize the function $I:C_E [0,\infty) \to [0,\infty]$ (that will depend on the generator of the fast process), with $C_E [0,\infty)$ the space of all continuous path $x: [0,\infty)\to E$, with $E= [0,\infty)^l$, such that 
\begin{equation}
    \mathbb{P}\left[X_N \approx x \right] \sim e^{-N I(x)}.
\end{equation}
This means that $X_N$ has a limit path $\tilde{x} \in C_{E}$ and this limit is the unique minimizer of the function $I$. Moreover, for any path $x\neq \tilde{x}$ such that $I(x)>0$, the probability that $X_N$ is close to $x$ is exponentially small.
Hence, characterizing the function $I$ means to characterize the limit behaviour of $X_N$. 

In Chapter 8 of \cite{FK06}, it is shown that the function $I$ is characterized by the unique solution of the Hamilton-Jacobi equation
\begin{equation}
    f-\lambda \bfH f = h \qquad (\text{or $\partial_t f - \bfH f = 0$}),
\end{equation}
with $\bfH f = \cH(f,\nabla f)$ and 
$\cH : E \times \R^l \to \R $ of the type
\begin{equation}\label{hamiltonian_example}
    \cH(x,p)=\sup_{\theta \in \Theta} \left[ \int_{F_M} V(y;x,p) \, d\theta(y) + \inf_{\varphi \in C^2(F_M)} \int_{F_M} e^{-\varphi}L_{x,p}e^{\varphi}\,d\theta \right].
\end{equation}
We refer to the Subsection \ref{subsection:example_Hamiltonian} for a detailed explanation on how to construct the Hamiltonian using a limiting procedure built on \eqref{eq:multiple-time: generator_extended}. In this context, $\Theta = \mathcal{P}(F_M)$, that is the probability measures over the set $F_M$, and 
\begin{align}
    V(y;x,p) &= \sum_{\gamma=(\gamma_x,\gamma_y)\in\Gamma_1} r(x,y,\gamma) (e^{\langle p,\gamma_x\rangle}-1) + \sum_{\gamma=(\gamma_x,\gamma_y)\in\Gamma_2} r(x,y,\gamma) (e^{\langle p, \gamma_x\rangle}-1),\\
    L_{x,p}f(x,y) &= \sum_{\gamma=(\gamma_x,\gamma_y)\in\Gamma_2} r(x,y,\gamma) e^{\langle p, \gamma_x\rangle} [ f(x,y+\gamma_y)- f(x,y)]\\ 
        &+ \sum_{\gamma=(\gamma_x,\gamma_y)\in\Gamma_3} r(x,y,\gamma) [f(x,y+\gamma_y) - f(x,y)],
\end{align}
where $r \in C^1(E \times \Gamma_i)$ are non negative smooth functions. 
Hence, in this context the well-posedness of the Hamilton-Jacobi equation plays a central role. 

In Subsection \ref{subsection:example_comparison}, we prove comparison principle for the equations $f - \lambda \bfH f = h$ and $\partial_t f - \bfH f = 0$, with $\bfH f = \cH (x,\nabla f(x))$ and $\cH$ as in \eqref{hamiltonian_example}, and existence of solutions for the stationary equation. In Subsection \ref{subsection:example_Hamiltonian}, we outline how the above Hamiltonian can be obtained from a limit procedure and an eigenvalue problem arising from the multi--scale Markov process described above. For these aims, we need two additional assumptions.
\begin{assumption}\label{assumption:irreducibility}
	The matrix $(R_x)_{y_1,y_2}= \sum_{k\in\{2,3\}} \sum_{\gamma \in \Gamma_k : y_2 = y_1 + \gamma_y} r(x,y_1,\gamma)$ is irreducible for every $x$.
\end{assumption}	

\begin{assumption}\label{assumption:inner_product} For all $z$ and $\gamma$ there exist  continuous functions $ \phi^{z,\gamma}_1,\phi^{z,\gamma}_2: [0,\infty)^l \to \R$ such that  $r(x, z, \gamma) = \phi^{z,\gamma}_1(x)\phi^{z,\gamma}_2(x)$ and such that
\begin{enumerate}
    \item  $\inf \phi_2^{z,\gamma}>0$;
    \item if $\langle \gamma_x , x-y \rangle >0 $ then $\phi_1^{z,\gamma} (x) < \phi_1 ^{z,\gamma}(y)$;
    \item if $\gamma_{x_i} < 0 $ and $x_i = 0$ then $\phi_1(x) = 0$.
\end{enumerate}
    
\end{assumption}
Assumption \ref{assumption:irreducibility} guarantees the existence of an eigenvalue of $V_{x,p} + L_{x,p}$ (i.e. step \ref{item: FengKurtz_eigenvalue} of Subsection \ref{subsection:example_Hamiltonian}). 

Assumption \ref{assumption:inner_product} restricts the possible rates and hence Hamiltonians of the type \eqref{hamiltonian_example} that we are able to treat with our theorems. First of all it removes cases of type $\cH(x,p)=x (e^p -1)$, with $x\in [0,\infty)$, for which it is well known that comparison principle fails (see e.g. Example E in \cite{ShWe05}). Moreover, it is straightforward to show Assumption \ref{assumption:inner_product} in many cases in one dimension, e.g. $\cH(x,p)=x(e^{-p}-1)$, $\cH(x,p)=x e^{\beta x} (e^{-p} - 1)$ in $[0,\infty)$.  It is also straightforward to verify the above assumptions for the two dimensional Example \ref{ex:mm}. 
Regarding higher dimensional cases, the assumption excludes examples in which there is an interaction between two molecules of the slow process. These cases, indeed, produce rates of the type $r_{z,\gamma}(x) = x_i x_j$ with $i,j \in \{1,\dots, l\}$ or such that $\langle \gamma_x, x- y \rangle > 0$ and not equal to $\phi_1(y) - \phi_1(x)$, for which Assumption \ref{assumption:inner_product} does not necessary hold. However, to our knowledge, examples of this type are largely unexplored. 

Cases in higher dimension for which the above assumption holds are e.g. Hamiltonians of the type $\cH(x,p) = x_1 (1+ x_2) (e^{p_1} -1)$.

\subsection{Comparison principle and existence of viscosity solutions}\label{subsection:example_comparison}
\begin{theorem}[Comparison principle]\label{theorem:comparison_example}
    Consider $\cH (x,p)$ as in \eqref{hamiltonian_example}. Suppose Assumption \ref{assumption:irreducibility} and Assumption \ref{assumption:inner_product}. Then comparison principles for $f-\lambda \cH(x, \nabla f(x)) = h$ and $\partial_t f - \cH(x,\nabla f(x)) = 0$ hold.
\end{theorem}

\begin{proof}
To prove the comparison principle we firstly mention that $\cH(x,p)$ is of the form \eqref{eq:results:variational_hamiltonian} with $\Theta = \mathcal{P}(F_M)$ and
 \begin{equation}\label{ex:Lambda}
\Lambda(x,p,\theta)= \int_{F_M} V(y;x,p) \, d\theta(y),    
 \end{equation}
 and 
 \begin{align}\label{ex:I}
 \cI(x,p,\theta)& = - \inf_{u\in C^2(F_M), \inf u>0} \int_{F_M} \frac{ L_{x,p} u } {u}\,d\theta= - \inf_{\varphi \in C^2(F_M)} \int_{F_M} e^{-\varphi} L_{x,p}e^{\varphi}\,d\theta.
 \end{align}
  We can then apply Theorem \ref{theorem:comparison_principle_variational} to show the comparison principle.
  In the following we verify Assumption \ref{assumption:regularity:Lambda-I}.
\begin{enumerate}[(I)]
    \item The function $p\mapsto \cH(x,p)$ is convex. Moreover, note that $V_{x,0}=0$. Hence, 
   $$\cH(x,0)= - \inf_\theta \cI(x,p,\theta) = 0,$$
   being $\cI \geq 0$ (see \cite{DoVa75a}) and there exists a measure $\theta^0_{x,p}$ such that $\cI(x,p,\theta^0_{x,p})=0$, due to Assumption \ref{assumption:irreducibility} (see \cite{Kl14} Theorem 17.51).
    \item $\theta \mapsto \Lambda(x,p,\theta)-\cI(x,p,\theta) $ is bounded for every $p$ and $x$ being a continuous function over a compact set.
    \item $\Upsilon(x)= \frac{1}{2} \log(1+ \sum_{i=1}^l x_i^2)$ is a containment function since the functions $$(r(x,y,\gamma))\left(e^{\frac{x}{x^2+1}\gamma_x}-1\right),$$ 
    $$e^{-\varphi}r(x,y,\gamma)e^{\frac{x}{1+x^2}\gamma_x}e^\varphi$$ and $$e^{-\varphi}r(x,y,\gamma)e^{\varphi}$$ are bounded for every $\gamma\in\Gamma$,  every $\varphi \in C^2(F_M)$ and every $(x,y)\in E\times F_M$.
    \item Let $(x,p)\in E \times\R^l$. The function $V(\cdot,x,p)$ is continuous and hence $(x,p,\theta)\mapsto\Lambda(x,p,\theta)$ is continuous. Moreover, $\cI$ is lower semicontinuous, as the supremum over continuous functions. Then, $\cI - \Lambda$ is lower semicontinuous and the first property of Definition \ref{def:Gamma-convergence} follows. We prove now that if $x_n \to x$ and $p_n \to p$ and for all $\theta \in \Theta$, there are $\theta_n$ such that $\theta_n \to \theta$ and 
    \begin{equation}\label{eq:gammaconv}
    \limsup_n \cI (x_n, p_n, \theta_n) \leq \cI(x,p,\theta).
    \end{equation}
    Then, the $\Gamma$-- convergence of $\cI - \Lambda$ will follow from \eqref{eq:gammaconv} and continuity of $\Lambda$. 

    For any $m \in \mathbb{N}$, there exists $\varphi_m \in C^2(F_M)$ such that 
    \begin{equation}
        \cI(x,p,\theta) \leq \int_{F_M} e^{-\varphi_m} L_{x,p} e^{\varphi_m} \, d\theta + \frac{1}{m}.
    \end{equation}
    Then, taking into account the continuity of $L_{x,p}$ and choosing $\theta_n = \theta$ for every $n$, we get
    \begin{equation}
        \limsup_n \cI(x_n,p_n,\theta_n) \leq \int_{F_M} e^{-\varphi_m} L_{x,p} e^{\varphi_m} \, d\theta + \frac{1}{m}.
    \end{equation}
    By letting $m$ to infinity we obtain \eqref{eq:gammaconv}.
    \item As $\Theta$ is compact, any closed subset of $\Theta$ is compact.
    \item As explained above, there exists a measure $\theta^0_{x,p}$ such that $\cI(x,p,\theta^0_{x,p})=0$. Then, $\cI(x,p,\theta^0_{x,p}) - \Lambda(x,p,\theta^0_{x,p}) \leq - \Lambda(x,p,\theta^0_{x,p})$. Taking $g(x,p)= - \Lambda(x,p,\theta^0_{x,p} )$, $\phi_{g}(x,p)$ is not empty, as $\theta_{x,p}^0 \in \phi_{g}(x,p)$.
    \item Let $(x_{\alpha,\eps},y_{\alpha,\eps},\theta_{\alpha,\eps})$ be a fundamental sequence as in Definition \ref{def:results:continuity_estimate}. Set $p_{\alpha,\eps} = \alpha (x_{\alpha,\eps} - y_{\alpha,\eps})$. We aim to show
    \begin{align}
       & \liminf_{\alpha \rightarrow \infty} (\Lambda - \cI)\left(x_{\alpha,\eps}, p_{\alpha,\eps}, \theta_{\alpha,\eps}\right) - (\Lambda-\cI)\left(y_{\alpha,\varepsilon},p_{\alpha,\eps},\theta_{\alpha,\eps}\right)\leq 0.
    \end{align}
  By the definition of $\Lambda$ and $\cI$ in \eqref{hamiltonian_example}, the difference above is of the type
  \begin{align}\label{eq:fundamentalseq}
  &\int_{F_M} \sum_{\gamma\in\Gamma_1 \cup \Gamma_2} \left(r(x_{\alpha,\eps},z,\gamma)-r(y_{\alpha,\eps},z,\gamma)\right)\left(e^{\langle p_{\alpha,\eps}, \gamma_x \rangle}-1 \right)\, d\theta +\\
  & \inf_{\varphi \in C^2(F_M)} \int_{F_M} e^{-\varphi} \left(\sum_{\gamma\in\Gamma_2} r(x_{\alpha,\eps},z,\gamma)-r(y_{\alpha,\eps},z,\gamma)\right) e^{\langle p_{\alpha,\eps}, \gamma_x \rangle} e^{\varphi}+\\
  & e^{-\varphi}\left(\sum_{\gamma\in \Gamma_3} r(x_{\alpha,\eps},z,\gamma) - r(y_{\alpha,\eps},z,\gamma)\right) e^{\varphi}\, d\theta.
  \end{align}
  Note that if $r(x,y,\gamma)$ is constant in $x$, the difference above is zero. Hence, we only take into account the  parameters $\gamma$ such that $r$ depends on $x$. 
  
  Moreover, by the upper bound \ref{item:def:continuity_estimate:3} in Definition \ref{def:results:continuity_estimate}, we find that there is some $\alpha(\eps)$ such that 
    \begin{equation}\label{eq:boundsupersolution}
        \sup_{\alpha\geq \alpha(\eps)} \Lambda(y_{\alpha,\eps}, p_{\alpha,\eps},\theta_{\alpha,\eps}) - \cI (y_{\alpha,\eps}, p_{\alpha,\eps},\theta_{\alpha,\eps}) <\infty.
    \end{equation}
    If $\lim_{\alpha} r(y_{\alpha,\eps},z,\gamma) > 0 $ for all $\gamma$, we can conclude by the bound \eqref{eq:boundsupersolution} that $e^{\langle p_{\alpha,\eps}, \gamma_x \rangle}$ is bounded. 
  Then, by property \ref{item:def:continuity_estimate:2} of Definition \ref{def:results:continuity_estimate} and continuity of the rates, \eqref{eq:fundamentalseq} converges to $0$ for $\alpha \to \infty$. 

  Consider now all terms $\gamma$ such that $r(y_{\alpha,\eps},z,\gamma) \to 0$ as $\alpha \to \infty$. 
Firstly note that, by Assumption \ref{assumption:inner_product}, \eqref{eq:fundamentalseq} equal to 
\begin{align}\label{eq:fundamentalseq_two}
    &\int_{F_M} \sum_{\gamma\in\Gamma_1 \cup \Gamma_2} \left(\phi^{z,\gamma}_1 (x_{\alpha,\eps})\phi^{z,\gamma}_2(x_{\alpha,\eps}) - \phi^{z,\gamma}_1 (y_{\alpha,\eps})\phi^{z,\gamma}_2(y_{\alpha,\eps})\right) \left(e^{\langle p_{\alpha,\eps}, \gamma_x \rangle}-1 \right)\, d\theta +\\
  & \inf_{\varphi \in C^2(F_M)} \int_{F_M} e^{-\varphi} \left(\sum_{\gamma\in\Gamma_2}  \left(\phi^{z,\gamma}_1 (x_{\alpha,\eps})\phi^{z,\gamma}_2(x_{\alpha,\eps}) - \phi^{z,\gamma}_1 (y_{\alpha,\eps})\phi^{z,\gamma}_2(y_{\alpha,\eps})\right) e^{\langle p_{\alpha,\eps}, \gamma_x \rangle} \right)e^{\varphi}+\\
  & e^{-\varphi}\left(\sum_{\gamma\in \Gamma_3} r(x_{\alpha,\eps},z,\gamma) - r(y_{\alpha,\eps},z,\gamma)\right) e^{\varphi}\, d\theta.
\end{align}

The last line converges to $0$ for $\alpha \to \infty$ by the continuity of the rates. 

If $\langle p_{\alpha,\eps} , \gamma_x \rangle <0$, $e^{\langle p_{\alpha,\eps} , \gamma_x \rangle}$ is bounded and the first two lines also converge to $0$ by continuity of the rates.

Consider the case $\langle p_{\alpha,\eps} , \gamma_x \rangle >0$. Then, by Assumption \ref{assumption:inner_product}, $\phi_1(y_{\alpha,\eps}) > \phi_1(x_{\alpha,\eps}) \geq 0$ and $ \phi_2 (y_{\alpha,\eps}) > 0$. Then, we can write the first two lines of \eqref{eq:fundamentalseq_two} as
\begin{align}
    &\int_{F_M} \sum_{\gamma\in\Gamma_1 \cup \Gamma_2} \underbrace{\left(\frac{\phi^{z,\gamma}_1(x_{\alpha,\eps}) \phi^{z,\gamma}_2 (x_{\alpha,\eps})}{\phi^{z,\gamma}_1(y_{\alpha,\eps}) \phi^{z,\gamma}_2 (y_{\alpha,\eps})} - 1 \right)}_{(1)}\underbrace{\left( \phi^{z,\gamma}_1(y_{\alpha,\eps}) \phi^{z,\gamma}_2 (y_{\alpha,\eps}) \right) \left(e^{\langle p_{\alpha,\eps}, \gamma_x \rangle}-1 \right)}_{(2)}\, d\theta +\\
  & \inf_{\varphi \in C^2(F_M)} \int_{F_M} e^{-\varphi} \left(\sum_{\gamma\in\Gamma_2}  \underbrace{\left(\frac{\phi_1^{z,\gamma}(x_{\alpha,\eps}) \phi^{z,\gamma}_2 (x_{\alpha,\eps})}{\phi_1^{z,\gamma}(y_{\alpha,\eps}) \phi^{z,\gamma}_2 (y_{\alpha,\eps})} - 1 \right)}_{(3)}\underbrace{\left( \phi^{z,\gamma}_1(y_{\alpha,\eps})\phi^{z,\gamma}_2(y_{\alpha,\eps}) \right) e^{\langle p_{\alpha,\eps}, \gamma_x \rangle}}_{(4)} \right)e^{\varphi}
\end{align}
For $\alpha \to \infty$, by Assumption \ref{assumption:inner_product} (2), $(1)$ and $(3)$ are negative and $(2)$ and $(4)$ are positive and bounded by \eqref{eq:boundsupersolution}. Then, for $\alpha$ big the second and third lines of \eqref{eq:fundamentalseq} are bounded above from zero and this concludes the proof.

\end{enumerate}
\end{proof}

\begin{proposition}[Existence of viscosity solutions]
    Consider $\cH(x,p)$ as in \eqref{hamiltonian_example}. Suppose Assumptions \ref{assumption:irreducibility} and \ref{assumption:inner_product}. Then, the function $R(\lambda)h$ defined in \eqref{resolvent} is the unique viscosity solution to $f-\lambda \bfH f = h$.
\end{proposition}
\begin{proof}
    We show that $\cH(x,p)$ in \eqref{hamiltonian_example} verifies Assumption \ref{assumption:Hamiltonian_vector_field}. Then, the result follows from Theorem \ref{theorem:existence_of_viscosity_solution}. We need to show that $\partial_p \cH(x,p) \cap T_E(x) \neq \emptyset$ for all $x\in E$.
We prove this in two steps: 
\begin{enumerate}
    \item We firstly show that for an Hamiltonian of the type \eqref{hamiltonian_example}, $\partial_p \left[\int_{F_M} V_{x,p} + e^{-\varphi} L_{x,p} e^{\varphi} \, d\theta \right] \subseteq \partial_p \cH(x,p)$;
    \item Secondly, we show that $\partial_p \left[\int_{F_M} V_{x,p} + e^{-\varphi} L_{x,p} e^{\varphi} \, d\theta \right] \cap T_E(x) \neq \emptyset $.
\end{enumerate}
\textit{Proof of step 1.} 
We call $\Psi_{x,\varphi,\theta} (p) = \int_{F_M} V_{x,p} + e^{-\varphi} L_{x,p} e^{\varphi} \, d\theta$. 

Firstly, recall that for a general convex function  $p \mapsto \Phi(p)$ we denote  $$\partial_p \Phi (p_0) = \{ \xi \in \R^l \, : \, \Phi(p) \geq \Phi(p_0) + (p-p_0)\cdot \xi  \quad \forall p\in \R^l\}.$$ 
Let $\xi \in \partial_p \Psi_{x,\varphi,\theta} (p)$ and $q \in \R^l$.
We call $\varphi_q$ the optimal map for $\cH(x,q)$ and $\theta_p$ the optimal measure for $\cH(x,p)$. Then we have
\begin{align}
    \cH(x,q) & \geq \int_{F_M} V_{x,q} + e^{-\varphi_q}L_{x,q} e^{\varphi_q} \, d\theta_p\\
    &\geq \int_{F_M} V_{x,p} + e^{-\varphi_q}L_{x,p} e^{\varphi_q} \, d\theta_p + (q-p)\cdot \xi \\
    &\geq \inf_\varphi \int_{F_M} V_{x,p} + e^{-\varphi}L_{x,p} e^{\varphi} \, d\theta_p + (q-p)\cdot \xi \\
    &= \cH(x,p) + (q-p) \cdot \xi
\end{align}
showing that $\xi \in \partial_p \cH(x,p)$ and hence that $\partial_p \Psi_{x,\varphi,\theta}(p)\subseteq \partial_p \cH(x,p)$.

\textit{Proof of step 2.} We claim that
      $\partial_p \Psi_{x,\theta,\varphi}(p) \cap T_E(x)$ is not empty for every $x\in E$. 
      Indeed, note that 
      \begin{equation}
          T_{[0,\infty)^l}(x) = \Pi_{i=1}^l T_i,
      \end{equation}
      with 
      \begin{equation}
      T_i =
      \begin{cases}
          \R \qquad &\text{if $x_i \neq 0$,}\\
           [0,\infty) \qquad &\text{if $x_i = 0$.}
     \end{cases}
    \end{equation}  
      Then, if $x_i \neq 0$, $\partial_{p_i} \Psi _{x_i, \varphi, \theta} (p) \subseteq T_i$ trivially. 

      If $x_i = 0$, by assumption \ref{assumption:inner_product}(3), if $\gamma_{x_i} <0$, $r(x,y,\gamma) = 0$. Hence, we can conclude that $\partial_{p_i} \Psi_{x_i, \varphi, \theta} (p) \geq 0$. And this conclude the proof.
\end{proof}

\subsection{Construction of the Hamiltonian with a Large deviations approach}\label{subsection:example_Hamiltonian}
In this subsection we show how the Hamiltonian in \eqref{hamiltonian_example} can be obtained from the study of the multi-scale Markov process arising in biological systems described above. 
In a large deviations context the aim is to characterize the function $I:C_E [0,\infty) \to [0,\infty]$, with $C_E [0,\infty)$ the space of all continuous path $x: [0,\infty) \to E$, such that 
\begin{equation}
    \mathbb{P}\left[X_N \approx x \right] \sim e^{-N I(x)}.
\end{equation}
In Chapter 8 of \cite{FK06}, it is shown that the function $I$ is characterized by the unique solution of the Hamilton-Jacobi equation
\begin{equation}
    f-\lambda \bfH f = h,
\end{equation}
with $\bfH f(x) = \cH(x,\nabla f(x))$ found in three steps:
\begin{enumerate}
    \item Given the \textit{generator} $A_N$ of the process $(X_N,I_N)$, define the \textit{non linear generator} $H_N f = \frac{1}{N}e^{-Nf}A_N e^{Nf}$ for $f$ such that $e^{Nf}\in D(A_N)$;
    \item Given $f\in C(E)$ and $h \in C(E\times F_M)$ define $f_N(x,y) = f(x) + N^{-1}h(x,y)$ and find $H_{h}$ such that $\lim_{N\to \infty} H_N f_N(x,y) = H_{h}(x,\nabla f(x), y)$;
    \item \label{item: FengKurtz_eigenvalue}For every $x$, find $h_x$ such that $H_{h_x}$ does not depend on $y$. Calling $p= \nabla f(x)$, this is equivalent to solve an eigenvalue problem : for all $x\in E$ and $p\in \R^l$, there exists $\cH(x,p)$ and $\bar{h}$ such that $H (x,p,y) \bar{h}(x,y) = \cH (x,p) \bar{h}(x,y)$.
\end{enumerate}

In this subsection, we show that in the analysis of the limit behaviour of the process $(X_N, Y_N)$, the three steps above give the Hamiltonian \eqref{hamiltonian_example} studied in the previous subsection.

In \cite{Po18}, the author studies the limit behaviour of a similar example following the method developed by Feng and Kurtz described in the steps above. However, in contrast with what it has been done in the above cited work, with our approach it is not necessary to find and calculate an explicit expression of the eigenvalue $\cH$ found in step \ref{item: FengKurtz_eigenvalue}. Therefore, our method presents an effective solution that can be applied to similar scenarios in which the eigenvalue problem can not be easily solved. 

We have the following result.
\begin{proposition}[Markov processes on multiple time-scale]\label{prop:multiple-time-Markov-processes}
    Consider the Markov process $(X_N,Y_N)$ having the operator \eqref{eq:generator_rescaled} as generator. Suppose Assumptions \ref{assumption:irreducibility} and \ref{assumption:inner_product}. Then, the following hold:
    \begin{enumerate}
        \item Let $f\in C(E)$ and $h\in C(E\times F_M)$. Define $f_N (x,y)= f(x) + N^{-1} h(x,y)$ and $H_N f =\frac{1}{N} e^{-Nf}A_N e^{Nf} $ provide $e^{Nf} \in D(A_N)$. Then,
        \begin{equation}
            \lim_N H_N f_N(x,y) = V(y; x, \nabla_xf(x))+ e^{-h(x,y)}L_{x,\nabla_xf(x)}e^{h(x,y)},
        \end{equation}
        with $$V(y;x,p) = \sum_{\gamma=(\gamma_x,\gamma_y)\in\Gamma_1} r(x,y,\gamma) (e^{\langle p, \gamma_x \rangle}-1) + \sum_{\gamma=(\gamma_x,\gamma_y)\in\Gamma_2} r(x,y,\gamma) (e^{\langle p , \gamma_x\rangle}-1),$$ and 
        \begin{align}
        L_{x,p}f(x,y) = &\sum_{\gamma=(\gamma_x,\gamma_y)\in\Gamma_2} r(x,y,\gamma) e^{\langle p, \gamma_x\rangle} [ f(x,y+\gamma_y)- f(x,y)]\\ 
        &+ \sum_{\gamma=(\gamma_x,\gamma_y)\in\Gamma_3} r(x,y,\gamma) [f(x,y+\gamma_y) - f(x,y)].
        \end{align}
        \item There exists a unique constant $\cH(x,p)$ and a unique function $g(x,y)$ solving the eigenvalue problem $(V(y;x,p) + L_{x,p})g(x,y)=\cH(x,p)g(x,y)$.
        \item Given the map $\bfH f = \cH(x,\nabla f(x))$, the Hamilton--Jacobi equation $f- \lambda \bfH f= h$ verifies the comparison principle.
    \end{enumerate}
\end{proposition}
\begin{proof}[Proof sketch.]
Recalling the generator $A_N$ in \eqref{eq:multiple-time: generator_extended}, note that the exponential generator $H_N$ acting on the test functions $f_N$ are
\begin{align}
    H_Nf_N(x,y)=& \sum_{\gamma \in \Gamma_1} r(x,y,\gamma) \left[ e^{N\left(f\left(x+\frac{1}{N}\gamma_x\right) - f(x)\right)+h\left(x+\frac{1}{N}\gamma_x,y\right)- h(x,y)}-1\right]\\
    &+ \sum_{\gamma\in\Gamma_2} r(x,y,\gamma) \left[e^{N\left(f\left(x+\frac{1}{N}\gamma_x\right)- f(x)\right)+h\left(x+\frac{1}{N}\gamma_x,y+\gamma_y\right)-h(x,y)}-1\right]\\
    &+\sum_{\gamma\in\Gamma_3} r (x,y,\gamma) \left[e^{h(x,y+\gamma_y)-h(x,y)}-1\right].
\end{align}
As a consequence, its limit is
\begin{align}
    \lim_N H_N f_N (x,y) = &\sum_{\gamma\in\Gamma_1} r(x,y,\gamma) [e^{\langle \nabla f(x), \gamma_x\rangle} - 1] + \sum_{\gamma\in\Gamma_2} r(x,y,\gamma) [e^{\langle \nabla f(x), \gamma_x\rangle} e^{h(x,y+\gamma_y)- h(x,y)}-1]\\
    & \sum_{\gamma\in\Gamma_3} r(x,y,\gamma) [ e^ {h(x,y+\gamma_y)- h(x,y)}-1] 
\end{align}
and the first claim is proven.

The second point follows from Assumption \ref{assumption:irreducibility} and Perron -- Frobenius Theorem. 

The third point follows from the fact that, being the eigenvalue of $V_{x,p} + L_{x,p}$, $\cH(x,p)$ is the Hamiltonian \eqref{hamiltonian_example} and from Theorem \ref{theorem:comparison_example} (see \cite{DoVa75a} for more details about the representation of the eigenvalue).
\end{proof}
\printbibliography
\end{document}